\documentclass[10pt,onecolumn,letterpaper]{article}

\usepackage{multirow}
\usepackage{hhline}

\usepackage{cvpr}
\usepackage{times}
\usepackage{epsfig}
\usepackage{graphicx}
\usepackage{amsthm}
\usepackage{amsmath,bm}
\usepackage{amssymb}
\usepackage{algorithm,algpseudocode}
\usepackage{booktabs}
\usepackage{subfigure}
\usepackage{epsfig}
\usepackage{wrapfig}
\usepackage[breaklinks=true,
bookmarks=ture,citecolor=red,    
colorlinks=true,
linkcolor=blue,
filecolor=magenta,      
urlcolor=cyan,]{hyperref}

\allowdisplaybreaks

\newtheorem{theorem}{Theorem}
\newtheorem{lemma}{Lemma}

\newtheorem{corollary}{Corollary}
\newtheorem{definition}{Definition}
\newtheorem{assumption}{Assumption}

\newtheorem*{lemma2}{Lemma}
\newtheorem*{corollary2}{Corollary}

\numberwithin{equation}{section}
\usepackage{datetime}

\def\r1{\color{black}}
\def\b1{\color{blue}}

\usepackage{fancyhdr}
\usepackage{authblk}
\cvprfinalcopy 
\def\cvprPaperID{****} 
\allowdisplaybreaks

\begin{document}
\lhead{}

\title{Stochastically Controlled Stochastic Gradient for the Convex and Non-convex Composition 
	problem}
\author[1]{ Liu Liu\thanks{lliu8101@uni.sydney.edu.au}}
\author[2]{ Ji Liu\thanks{ji.liu.uwisc@gmail.com}}
\author[3]{ Cho-Jui Hsieh\thanks{chohsieh@ucdavis.edu}}
\author[1]{ Dacheng Tao\thanks{dacheng.tao@sydney.edu.au}}
\affil[1]{UBTECH Sydney AI Centre and SIT, FEIT, The University of Sydney}
\affil[2]{Department of Computer Science,	University of Rochester}
\affil[3]{University of California, Davis}
\maketitle
\maketitle
\thispagestyle{fancy}

\begin{abstract}
	In this paper, we consider the convex and non-convex composition problem with the structure 
	$\frac{1}{n}\sum\nolimits_{i = 1}^n {{F_i}( {G( x )} )}$, where $G( x ) 
	=\frac{1}{n}\sum\nolimits_{j = 1}^n {{G_j}( x )} $ is the inner function, and $F_i(\cdot)$ is 
	the 
	outer function.	We explore the  
	variance reduction based  method to solve the composition optimization. Due to the fact that 
	when the number of inner function and outer function are large, it is not reasonable to 
	estimate them directly, thus we apply the 
	stochastically controlled stochastic gradient 
	(SCSG) 	method to estimate the gradient of the composition function and the value of the inner 
	function. 	
	The query complexity of our proposed method  for the convex  and  non-convex problem is equal 
	to or better than the current method for the composition  problem.
	Furthermore, we also present the mini-batch version of the proposed method, which has the 
	improved the query complexity with related to the size of the mini-batch.
\end{abstract}

\section{Introduction}
In this paper, we study the problem of the following non-convex composition 
minimization
{\small
\begin{align}\label{VRNonCS-SCSG:ProblemMainComposition}
\mathop {\min }\limits_{x \in {\mathbb{R}^N}} \left\{ {f( x )\mathop  = \limits^{{\rm{def}}} F( {G( 
		x )} ) \mathop  = \limits^{{\rm{def}}} \frac{1}{n}\sum\limits_{i = 1}^n {{F_i}\left( 
		{\frac{1}{n}\sum\limits_{j = 1}^n {{G_j}( x )} } \right)} } \right\},
\end{align}}
where $f$: $\mathbb{R}^N$ $\to$ $\mathbb{R}$ is a non-convex function,  each $F_i$: $\mathbb{R}^M 
\to \mathbb{R}$ is a smooth function, each  $G_i$: $\mathbb{R}^N$$\to$ $\mathbb{R}^M$ is a mapping 
function, $n$ is the number of $F_i$'s and $G_j$'s.  We call 
$G(x)$:$=\frac{1}{n}\sum\nolimits_{j=1}^nG_j(x)$  the inner function, and $F( {w} )$:$ = 
\frac{1}{n}\sum\nolimits_{i = 1}^n {{F_i}( { w} )} $ the outer function. There are many machine  
learning application such as such as reinforcement learning 
\cite{sutton1998reinforcement,wang2017stochastic,wang2016accelerating}
and nonlinear 
embedding 
\cite{hinton2003stochastic,dikmen2015learning}, that can be formed 
to the composition problem with two finite-sum structure $\frac{1}{n}\sum\nolimits_{i = 1}^n 
{{F_i}( {\frac{1}{n}\sum\nolimits_{j = 1}^n {{G_j}( x )} } )} $. For example,
	{ \begin{align*}
	\mathop {\min }\limits_x {\| {\mathbb{E}[ B ]x - \mathbb{E}[ b ]} \|^2},
	\end{align*}
	where $\mathbb{E}[ B ] = I - \gamma {P^{\pi} }$, $\gamma\in(0,1)$ is a discount factor, 
	$P^{\pi}$ is the transition probability,	$\mathbb{E}[ b ] = {r^{\pi} }$, and $r^{\pi}$ 
	is the expected state transition reward.} Another example is the mean-variance in 
risk-averse learning:
\begin{align*}
\mathrm{min}_x\, \mathbb{E}_{a,b}[ h( {x;a,b} ) ] + \lambda \mathrm{Var}_{{a,b}}[ {h( {x;a,b} 
	)} ],
\end{align*}
where $h(x;a,b)$ is the  loss function with random variables $a$ and $b$. $\lambda>0$ is a 
regularization parameter. Stochastic neighbour embedding (SNE)  
\cite{hinton2003stochastic}
is the non-convex problem that  map 
data from a high dimensional space to a 
low dimensional space. 
\begin{align*}
\mathop {\min }\nolimits_x \sum_t^{} {\sum_i^{} {{p_{i|t}}\log 
		\frac{{{p_{i|t}}}}{{{q_{i|t}}}}} } ,
\end{align*}
where 
\begin{align*}
{p_{i|t}} = \frac{{\exp ( - {{\left\| {{z_t} - {z_i}} \right\|}^2}/2\sigma 
		_i^2)}}{{\sum\nolimits_{j \ne t} {\exp ( - {{\left\| {{z_t} - {z_j}} \right\|}^2}/2\sigma 
			_i^2)} 
}}, \quad {q_{i|t}} = \frac{{\exp ( - {{\left\| {{x_t} - {x_i}} 
				\right\|}^2})}}{{\sum\nolimits_{j 
			\ne t} {\exp ( - {{\left\| {{x_t} - {x_j}} \right\|}^2})} }}, 
\end{align*}
and $\sigma_i$ is {the predefined parameter to control the sensitivity to the distance.}
$\{z_i\}_{i=1}^n$ and $\{x_i\}_{i=1}^n$ denote the representation of $n$ 
data points in the high dimensional space and the low dimensional space, respectively.


Recently, there many stochastic optimization methods solving the composition problem, such as 
stochastic 
gradient method \cite{wang2017stochastic,wang2016accelerating} and the variance-reduction based 
method 
\cite{lian2016finite,liu2017duality,liu2017variance}. However, there are two main problems 
encountered in the composition function: 1) the inner function $G(x)$ is the finite-sum structure. 
When the number of $G_i(x)$ is large, it will need more computation cost; 2) if the inner function 
$G(x)$ is estimated, the expectation of the stochastic gradient $f(x)$ with respect to $i_k,j_k\in 
[n]$ is 
not 
equal to the $\nabla f(x)$. That is 
\begin{center}
	$\mathbb{E}_{i_k,j_k}[ ( \partial {G_{j_k}} (x))^\mathsf{T}\partial {F_{i_k}}( \tilde G(x)) ] 
	\ne 
	\nabla f( x 
	)$,
\end{center}
where $\tilde G(x)) $is the estimation of $G(x)$, $\partial {G_{j_k}}$  is the partial gradient of 
$G_{j_k}(x)$. Furthermore, we use the  query complexity to evaluate the algorithm, that is the 
 number of {component function} queries used to compute the gradient.

Stochastic gradient method, such as Stochastic compositional gradient descent (SCSG) 
\cite{wang2017stochastic} estimates the inner 
function $G(x)$ by an iterative weighted average of the past values of the $G(x)$, then perform the 
stochastic quasi-gradient iteration. The advantage of this method is that it does not depend on $n$ 
but 
with poor query complexity to the desired point.  Variance-reduction method such as 
Compositional-SVRG \cite{lian2016finite}  
 estimates the inner function $G(x)$ and the gradient of function 
$f(x)$ by using the finite-sum structures, which deriving the linear convergence rate with the 
relationship of $n$. Table \ref{VRNonCS-SCSG:Table} present the query complexity result with 
different algorithms. 

\begin{table*}[t]
	\centering
	\begin{tabular}{|l|c|c|}
		\hline
		{Algorithm}  
		&{Strongly Convex}
		&Non-convex \\      
		\hline
		SCGD \cite{wang2017stochastic}
		&{$\mathcal{O}(1/\varepsilon^{3/2})$}
		&{$\mathcal{O}(1/\varepsilon^{4})$}
		\\
		\hline
		Acc-SCGD \cite{wang2017stochastic}
		&{$\mathcal{O}(1/\varepsilon^{5/4})$} 
		&{$\mathcal{O}(1/\varepsilon^{7/2})$} 
		\\
		\hline
		ASC-PG \cite{wang2016accelerating}
		&{$\mathcal{O}(1/\varepsilon^{5/4})$}
		&{$\mathcal{O}(1/\varepsilon^{9/4})$}
		\\
		\hline
		SC-SVRG\cite{lian2016finite}\cite{liu2017variance} 
		&  ${\cal O}\left( {\left( {n + L_f^2/{\mu ^4}} \right)\log \left( {1/\varepsilon } 
			\right)} \right).$
		&{$\mathcal{O}( n^{4/5}/\varepsilon
			)$} 
		\\
		\hline
		SC-SCSG 
		&  ${\cal O}\left( {\left( {\min \left\{ {n,\frac{1}{{\varepsilon {\mu ^2}}}} \right\} + 
				\frac{{L_f^2}}{{{\mu ^2}}}\min \left\{ {n,\frac{1}{{{\mu ^2}}}} \right\}} 
				\right)\log 
			\left( {1/\varepsilon } \right)} \right)$
		&{$\mathcal{O}( \min \{ 
			{{1}/{{{\varepsilon 
							^{9/5}}}},{{{n^{4/5}}}}/{\varepsilon }} 
			\} )$} 
		\\
		\hline
	\end{tabular}%
	\vspace{5pt}
	\caption{Comparison of the query complexity with different algorithms }
	\label{VRNonCS-SCSG:Table}%
\end{table*}%

Motivated by the recent work \cite{lei2017less,lei2017non,AllenZhu2017-natasha2} that the 
convergence rate of the finite-sum structure function has the general result under the relationship 
between $n$ and $\varepsilon$. Here, we use $\varepsilon$ to evaluate the terminal of the convex 
and non-convex function by $f\left( x 
\right) - f\left( {{x^*}} \right) \le \varepsilon$ and 
${\left\| {\nabla f\left( x \right)} \right\|^2} \le \varepsilon$, respectively, where $x^*$ is the 
optimal point 
in the convex function. The core aspect of these kinds of algorithms is similar to the stochastic 
variance-reduced gradient (SVRG) that using a snapshot vector to compute the ``gradient" of the 
function. The difference lies that the gradient is no longer computed directly but rather using the 
random subset, called   stochastically controlled stochastic gradient (SCSG). We explore the SCSG 
based method to the composition problem with both convex and non-convex function and analyze the 
corresponding the convergence and query complexity.

In this paper, we develop a novel stochastic composition optimization through  stochastically 
controlled stochastic gradient (SC-SCSG) method to two finite-sum structure. The main contributions 
are summarized below:
\begin{itemize}
	\item We provide the variance reduction based method to estimate the inner function $G(x)$. 
	Similar to the SCSG that estimate the gradient, the function $G(x)$ can also be estimated by a 
	snapshot $\tilde{x}_s$, in which  $G(\tilde{x}_s)$ is not computed directly, but rather based 
	on the random subset from $[n]$.  We also analyze the size of the subset such that can lead to 
	the desired precision for both convex and non-convex function. 
	\item 	After obtaining the estimated inner function, we consider the gradient of the function 
	$f(x)$. Here, we can also apply the SCSG based method to estimate the gradient. However, there 
	are 
	two situations 	encountered in the estimate process. 1) the expectation of the gradient is no 
	longer the 	unbiased estimation. 2) the gradient of $f(\tilde{x}_s)$ at the snapshot is 
	formed by two 	random subsets, which are used for  the function $F_i$ and $G_j$ respectively. 
	Nevertheless, 
	we also provide the bound of the subset size  that we can use the estimated gradient to update 
	the iteration. The details analysis can be referred to Section 
	\ref{VRNonCS-SCSG:Section:SCSG-convex}.
	\item The mini-batch version of the proposed algorithm is also provided for both the convex and 
	non-convex function. The corresponding query complexities are improved based on the size of the 
	mini-batch. More information can be referred to Section 
	\ref{VRNonCS-SCSG:Section:SCSG-minibatch}.
\end{itemize}
\subsection{Results}
We give the general query complexity  of the composition problem based on SCSG based method. The 
results present us an intuitive explanation for comparing with other algorithms. Note that the 
Algorithm \ref{VRNonCS-SCSG:AlgorithmI} can be used to both convex and non-convex problems that 
deriving the corresponding query complexities. Furthermore, Algorithm 
\ref{VRNonCS-SCSG:AlgorithmII} present the mini-batch version of the proposed method.

\textbf{{Convex function}} The query complexity for the convex function is 
\begin{center}
	$\mathcal{O}\left( 
	{\left( {\min \left\{ {n,\frac{1}{{{\varepsilon\mu ^2}}}} \right\} + \frac{{L_f^2}}{{{\mu 
						^2}}}\min 
			\left\{ {n,\frac{1}{{{\mu ^2}}}} \right\}} \right)\log \left( {1/\varepsilon } 
		\right)} 
	\right)$,
\end{center}
where $\mu$ is the constant of strongly convex of $f(x)$. The result is the same as that of 
\cite{lian2016finite} if $n\le 1/(\varepsilon\mu^2)$

\textbf{Non-convex function} The query complexity is
	{$\mathcal{O}( \min \{ 
		{{1}/{{{\varepsilon 
						^{9/5}}}},{{{n^{4/5}}}}/{\varepsilon }} 
		\} )$}, which can be better than that of \cite{wang2016accelerating} and comparable 
	to that of \cite{liu2017variance}.

\textbf{{Mini-batch}} For the mini-batch version, the query complexity  can be improved  to some 
extent comparing with above results, that is {$\mathcal{O}( \min \{ 
	{{1}/{{{\varepsilon 
					^{9/5}}}},{{{n^{4/5}}}}/{\varepsilon }} 
	\} /b^{1/5})$} and
\begin{center}
	$\mathcal{O}\left( 
	{\left( {\min \left\{ {n,\frac{1}{{{\varepsilon\mu ^2}}}} \right\} + \frac{{L_f^2}}{{{b\mu 
						^2}}}\min 
			\left\{ {n,\frac{1}{{{\mu ^2}}}} \right\}} \right)\log \left( {1/\varepsilon } 
		\right)} 
	\right)
	$,
\end{center} 
 for convex and non-convex function.
\subsection{Related work}
As the data increase, stochastic optimization  has been the popular method in machine learning and 
deep learning, especially for the finite-sum function. The typical algorithm include (stochastic 
gradient descent) SGD \cite{ghadimi2016accelerated}, stochastic variance reduction gradient 
(SVRG)\cite{johnson2013accelerating,reddi2016stochastic}, 
stochastic dual coordinate 
ascent (SDCA) \cite{shalev2014accelerated,shalev2013stochastic} and the  accelerated 
method Nesterov's 
method \cite 
{nesterov2013introductory},  accelerated randomized proximal coordinate (APCG) 
\cite{lin2014accelerated,lin2014acceleratedSIAM} and  Katyusha  method \cite{allen2016katyusha}.  
As 
the function is finite-sum structure, 
the general process for optimization is randomly selected one or a block component function to 
estimate the gradient. Thus the estimated gradient leads to the large variance of the gradient. 
Variance reduction method estimates the gradient by using a  snapshot in which the gradient of the 
function is computed at this point, which can appropriately reduce the variance.

The composition function can also be solved by using above algorithms, however, two 
finite-sum structures prevent implementation directly due to the fact that the computation of the 
inner function may increase the query complexity. Recently, Wang \textit{et al.} 
\cite{wang2017stochastic} first proposed the first-order stochastic compositional gradient methods 
(SCGD)  to solve  such  problems, which used two steps to alternately update the variable and inner 
function.  The SCGD method has the query complexity 
$\mathcal{O}(\varepsilon^{-7/2})$ for the general function and $\mathcal{O}(\varepsilon^{-5/4})$ 
for the strongly 
convex function. 
Liu \textit{et al.} \cite{wang2016accelerating} employed  Nesterov's method to accelerate the 
composition problem with $\mathcal{O}(\varepsilon^{-5/4})$ and $\mathcal{O}(\varepsilon^{-9/4})$ 
for strongly convex and non-convex function.  However, these methods estimate the inner function 
by an iterative weighted average of the past function. Such estimation did not take advantage of 
the finite-sum structure.

Based on the variance reduction technology, Lian \textit{et al.} \cite{lian2016finite} first 
applied the SVRG-based 
method to estimate the inner function $G(x)$ and the gradient of the function $f(x)$ as well. The 
linear convergence rate is obtained. In the following, 
Liu \textit{et. al} \cite{liu2017duality} apply the duality-free method to the composition problem 
and derive the linear convergence rate as well.  Yu and 
Huang \cite{yu2017fast}  applied the  ADMM-based \cite{ boyd2011alternating} method  and provide an 
analysis of the convex function without requiring Lipschitz smoothness. Moreover, Liu \textit{et. 
al} \cite{liu2017variance} considered the non-convex function and analyzed the query complexity 
with 
different sizes of the inner function and outer function. The details of the query complexity are 
provided.

There are many recent papers considering the variance reduced method that estimates the gradient 
using the random subset rather than computing directly. 
Lei and Jordan \cite{lei2017less} proposed an SCSG method to the convex finite-sum function, and 
then applied to the non-convex problem in \cite{lei2017non} that using less than a single pass to 
compute the gradient at the snapshot point. In the following, Allen-Zhu 
\cite{AllenZhu2017-natasha2} also proposed Natasha1.5 algorithm, in which the gradient for each 
epoch is based on the random subset. Moreover, the objective function has the regularization term. 
Liu \textit{et. al} \cite{liu2018stochastic} applied the SCSG based method to the zeroth-order 
optimization with the finite-sum function.

The rest of paper is organized as follows: In section \ref{VRNonCS-SCSG:Section:Preliminaries}, we 
give preliminaries used for analyzing the proposed algorithm. Section 
\ref{VRNonCS-SCSG:Section:SCSG} presents the SCSG-based method for the composition problem. we give 
the convergence and query complexity for the convex and non-convex function in Section 
\ref{VRNonCS-SCSG:Section:SCSG-convex} and Section \ref{VRNonCS-SCSG:Section:SCSG-Non-convex}, 
respectively. Section \ref{VRNonCS-SCSG:Section:SCSG-minibatch} gives the mini-batch version. 
We conclude our paper in Section \ref{VRNonCS-SCSG:Section:Conclusion}.

\section{Preliminaries}\label{VRNonCS-SCSG:Section:Preliminaries}
Throughout this paper, we use the Euclidean norm denoted by $\|\cdot\|$. We use $i \in [ n ]$ and 
$j 
\in [ m ]$ to denote that $i$ and $j$ are generated from $[ n ] = \{ {1,2,...,n} \}$, and $[ m ] = 
\{ {1,2,...,m} \}$. We denote by ${( {\partial G( x )} )^\mathsf{T}}\nabla F( {G( x )} )$ the full 
gradient of the function $f$,  $\partial G( x )$ the partial gradient of $G$, and  ${( {\partial 
		G_{j_k}( x )} )^\mathsf{T}}\nabla F_{i_k}( {G( x )} )$ as the stochastic gradient of the 
		function 
$f$, where $i_k$ and $j_k$ are randomly selected from $[n]$ and $[m]$.  We 
use $\mathbb{E}$ to denote the expectation. 
Note that all the variable such as subset $\cal A$ and $\cal B$, element $i_k$ and $j_k$ are 
independently selected from $[n]$ or $[m]$, in particular, the element in $\cal 
A$ and $\cal B$ are independent. So we  use  $\mathbb{E}$ in instead of 
$\mathbb{E}_{i_k}$,$\mathbb{E}_{j_k}$ ,$\mathbb{E}_{\cal A}$  and 
$\mathbb{E}_{\cal B}$ except particular stated. We use  $A=|\cal A|$ to denote the number of the 
elements in the set $\cal D$ and define ${G_\mathcal{A}(x)} = \frac{1}{A}\sum_{1 \le 
	j \le A} {{G_{\mathcal{A}\left[ j \right]}(x)}} $. 
Recall two definitions on Lipschitz function and smooth function.
\begin{definition} \label{VRNonCS-SCSG:DefinitionLipschitzFunction}
	A function $p$ is called a Lipschitz function on $ \mathcal{X}$ if there is a constant $B_p$ 
	such that $	\| {p( x ) - p( y )} \| \le {B_p}\| {x - y} \|$, $\forall x,y\in \mathcal{X}$.
\end{definition}
\begin{definition}
	A function $p$ is called a $L_p$-smooth function on $ \mathcal{X}$ if there is a constant $L_p$ 
	such that $\| {\nabla p( x ) - \nabla p( y )} \| \le {L_p}\| {x - y} \|$, and equal to $p( y ) 
	\le p( x ) + \langle {\nabla p( x ),y - x} \rangle  + {L_p}/2{\| {y - x} \|^2}$, $\forall 
	x,y\in \mathcal{X}$.
\end{definition}
We make the following assumptions used for the  discussion of the convergence rate and complexity 
analysis.
\begin{assumption}\label{VRNonCS-SCSG:Assumption:f-strong}
	For function $f$: ${\mathbb{R}^M} \to {\mathbb{R}}$, all $i\in [n]$,  
	\begin{itemize}
		\item $F$ is $\mu$-strongly convex satisfying $f\left( y \right) \ge f\left( x \right) + 
		\langle f\left( x \right),y - x\rangle  + \frac{\mu }{2}E{\left\| {x - y} \right\|^2}$.
		\item  $f$ has the optimal point $x^*$, then $\langle f\left( {{x_k}} \right),{x^*} - 
		{x_k}\rangle  \le  - \mu {\left\| {{x_k} - {x^*}} \right\|^2}$.
	\end{itemize}
\end{assumption}
\begin{assumption}\label{VRNonCS-SCSG:Assumption:G}
	For function $G_j$: ${\mathbb{R}^N} \to {\mathbb{R}^M}$, all $j\in [m]$,  
	\begin{itemize}
		\item $G_j$ has the bounded Jacobian with a constant ${B_G}$, that is $\| {\partial G_j( x 
			)} \| \le {B_G}$, $\forall  x \in {\mathbb{R}^N}$,  then  $ G_j( x )$ is also a 
			Lipschitz 
		function that satisfying $	\| {G_j( x ) - G_j( y )} \| \le {B_G}\| {x - y} \|$,  $\forall  
		x,y \in {\mathbb{R}^N}$.
		\item $G_j$ is $L_G$-smooth satisfying $	\| {\partial G_j( x ) - \partial G_j( y )} \| 
		\le {L_G}\| {x - y} \|$,  $\forall  x,y \in {\mathbb{R}^N}$.
	\end{itemize}
\end{assumption}
\begin{assumption}\label{VRNonCS-SCSG:Assumption:F}
	For function $F_i$: ${\mathbb{R}^M} \to {\mathbb{R}}$, all $i\in [n]$,  
	\begin{itemize}
		\item $F_i$ has the bounded gradient with a constant ${B_F}$,  that is $\|{\nabla F_i( y )} 
		\| \le {B_F}$, $\forall  y \in {\mathbb{R}^M}$.
		\item  $F_i$ is $L_F$-smooth satisfying $	\| {\nabla F_i( x ) - \nabla  F_i( y )} \| \le 
		{L_F}\| {x - y} \|$,  $\forall  x,y \in {\mathbb{R}^M}$.
	\end{itemize}
\end{assumption}

\begin{assumption} \label{VRNonCS-SCSG:Assumption:GF}
	For function $F_i(G(x))$: ${\mathbb{R}^N} \to {\mathbb{R}}$, all $i\in [n]$,  there exist a 
	constant $L_f$ satisfying 	
	\begin{align*}
	\| {{( {\partial {G_j}( x )} )^\mathsf{T}}\nabla {F_i}( {G( x )} ) - {( {\partial {G_j}( y )} 
			)^\mathsf{T}}\nabla {F_i}( {G( y )} )} \|\nonumber \le {L_f}\| {x - y} \|,\forall j 
	\in [m], 
	\forall  
	x,y \in {\mathbb{R}^N}.
	\end{align*}		
\end{assumption}

\begin{assumption} \label{VRNonCS-SCSG:AssumptionIndependent}
	We assume that $i_k$ and $j_k$ are independently and randomly selected from $[n]$ and $[m]$, 
	$z\in \mathbb{R} ^M, x\in \mathbb{R}^N$, 
	\begin{align*}
	\mathbb{E}[ {{( {\partial {G_{j_k}}( x )} )^\mathsf{T}}\nabla {F_{i_k}}( z )} ] = 
	{( {\partial G( x )} )^\mathsf{T}}\nabla F( z ), 
	\end{align*} 	
\end{assumption}

\begin{assumption}\label{VRNonCS-SCSG:Assumption:Middle-Bound}
	We assume that $H_1$ and $H_2$ are the upper bounds on the variance of the functions $G(x)$ and 
	$(\partial G(x))^\mathsf{T}\nabla F(y)$, respectively, that is,
	{
	\begin{align*}
	\frac{1}{n}\sum\limits_{i = 1}^n& {{{\left\| {G(x) - {G_i}(x)} \right\|}^2}}  \le {H_1}.\\
	\frac{1}{{n^2}}\sum\limits_{j = 1}^{n} \sum\limits_{i = 1}^n& {{{\left\| {{(\partial 
						G(x))^\mathsf{T}}\nabla F(y) - {(\partial {G_j}(x))^\mathsf{T}}\nabla 
						{F_i}(y)} \right\|}^2}}  
	\le {H_2}.	
	\end{align*}
}
\end{assumption}

In the paper, we denote by $x_k^s$ the $k$-th inner iteration at $s$-th epoch. But in each epoch 
analysis, we drop the superscript $s$ and denote by $x_k$ for $x_k^s$ . We let $x^*$ be the optimal 
solution of $f(x)$. Throughout the convergence analysis, we use  $\mathcal{O}( \cdot )$ notation to 
avoid many constants, such as $B_F$, $B_G$, $L_F$, $L_G$ and $L_f$,... that are irrelevant with the 
convergence rate  and provide insights to analyze the iteration and query complexity.

\section{Stochastic Composition via SCSG for the composition problem} 
\label{VRNonCS-SCSG:Section:SCSG}
\begin{algorithm*}[t]
	\caption{SC-SCSG for the  composition problem}
	\label{VRNonCS-SCSG:AlgorithmI}
	\begin{algorithmic}
		\Require $K$, $S$, $\eta$ (learning rate), $\tilde{x}_0$ and $\mathcal{D} = \left[ 
		{{\mathcal{D}_1},{\mathcal{D}_2}} \right]$
		\For{$s =0,1, 2,\cdots,S-1$} 
		\State Sample from $[n]$ for D times to form mini-batch $\mathcal{D}_1$ 
		\State Sample from $[n]$ for D times to form mini-batch $\mathcal{D}_2$ 		
		\State $\nabla {\hat f_{ \mathcal{D}}}({{\tilde x}_s}) = {(\partial G_{ 
				\mathcal{D}_1}({{\tilde 
					x}_s}))^\mathsf{T}}{\nabla {F_{{ \mathcal{D}_2}}}(G_{ 
				\mathcal{D}_1}({{\tilde 
					x}_s}))}  
		$\Comment{D Queries}
		\State $x_0=\tilde{x}_s$
		\For{$k =0,1,2,\cdots,K-1$}
		\State
		Sample from $[m]$  to form mini-batch $\mathcal{A}$ 
		\State
		${{\hat G}_k} = {G_{{{\cal A}}}}({x_k}) - {G_{{{\cal A}}}}({{\tilde x}_s}) + 
		{G_{\mathcal{D}_1}}({{\tilde 
				x}_s})$\Comment{A Queries}
		\State Uniformly and randomly pick $i_k$ and $j_k$ from $[n]$
		\State
		Compute the estimated gradient $\nabla {{\tilde f}_k}$ from 
		(\ref{VRNonCS-SCSG:Definition:Estimate-f})
		\Comment{4 Queries}
		\State
		${x_{k+1}}= {x_k} - {\eta} \nabla {{\hat f}_k} $
		\EndFor
		\State Update $\tilde{x}_{s+1}=x_K $
		\EndFor \\	
		\textbf{Output:}  $\hat x_k^s$ is uniformly and randomly chosen from  $s\in\{0,...,S-1\}$ 
		and $k\in \{0,..,K-1\}$.
	\end{algorithmic}
\end{algorithm*}

In this section, we present the variance-reduction based method for the composition problem, which 
can be used for both the convex and non-convex function. Before describing the proposed algorithm, 
we 
recall the original SVRG \cite{johnson2013accelerating}. The general process of the SVRG works as 
follows. The update process is divided into $S$ epochs, each of the epoch 
consists of K iterations. At the beginning of each epoch, SVRG define a snapshot vector 
$\tilde{x}_s$, and then compute the full gradient $\nabla 
f(\tilde{x}_s)$. In the inner iteration of the current epoch, SVRG defines the estimated gradient 
by 
randomly selecting $i_k$ from [n] at the $k$-th iteration,
\begin{align}\label{VRNonCS-SCSG:Definition:Estimate-f-svrg}
{(\partial G({x_k}))^\mathsf{T}}\nabla {F_{{i_k}}}({G_k}) - {(\partial 
	G({{\tilde 
			x}_s}))^\mathsf{T}}\nabla {F_{{i_k}}}(G({{\tilde x}_s})) + \nabla f({{\tilde x}_s}).
\end{align} 

However, for the composition problem, there are also variance-reduction based methods in
\cite{lian2016finite}, \cite{liu2017duality} and 
\cite{liu2017variance}. The difference with SVRG is that there is another estimated function for 
$G(x)$, as $G(x)$ is also the finite-sum structure. There methods defined the estimate function as 
\begin{align}\label{VRNonCS-SCSG:Definition:Estimate-G-svrg}
{{\tilde G}_k} = {G_{{{\cal A}}}}({x_k}) - {G_{{{\cal A}}}}({{\tilde x}_s}) + 
{G}({{\tilde 
		x}_s}),
\end{align}
where  $\mathcal{A}$ is the mini-batch formed by randomly sampling from $[n]$. Whereas, as the 
number of the inner function $G_j$ and the outer function $F_i$ increase, it is not reasonable to 
compute the full gradient of $f(x)$ and the full function $G(x)$ directly for each epoch. 

Extended from the  
SCSG \cite{lei2017non}\cite{lei2017less} and Natasha1.5 \cite{AllenZhu2017-natasha2}, we 
present a new algorithm for the composition problem  as shown in Algorithm 
\ref{VRNonCS-SCSG:AlgorithmI}. First of all, we introduce the two subset $\mathcal{D}_1$ and 
$\mathcal{D}_2$, which are independent with each other and formed by randomly selecting from $[n]$, 
respectively. We define  $\mathcal{D} = \left[ 
{{\mathcal{D}_1},{\mathcal{D}_2}} \right]$ for a new variable. $\mathcal{D}_1$ is used for 
estimating the inner function. Based on the variance 
reduction technology, the estimated inner function  at $k$-th 
iteration of $s$-th epoch is
\begin{align}\label{VRNonCS-SCSG:Definition:Estimate-G}
{{\hat G}_k} = {G_{{{\cal A}}}}({x_k}) - {G_{{{\cal A}}}}({{\tilde x}_s}) + 
{G_{\mathcal{D}_1}}({{\tilde 
		x}_s}),
\end{align}
where the subset of $\mathcal{A}$ is the same as in 
(\ref{VRNonCS-SCSG:Definition:Estimate-G-svrg}). Note that $\cal A$ and $\cal D$ are independent 
with each other. The difference with (\ref{VRNonCS-SCSG:Definition:Estimate-G-svrg}) is computing 
the third them that is 
under the 
subset $\mathcal{D}_1$ rather than $[n]$ as in (\ref{VRNonCS-SCSG:Definition:Estimate-G-svrg}). 
Throughout the paper, we assume that $|\cal A|\le |\cal D|$. $\mathcal{D}_2$ is use to estimate the 
outer function $F$. The key distinguish with  \cite{lei2017non, lei2017less,AllenZhu2017-natasha2} 
is the biased full gradient of $f(\tilde{x}_s)$. We define this estimated full gradient of 
$f(\tilde{x}_s)$ for each  epoch as
$\nabla {\hat f_{ \mathcal{D}}}({{\tilde x}_s}) = {(\partial G_{\mathcal{D}_1}({{\tilde 
			x}_s}))^\mathsf{T}}\nabla {F_{{ \mathcal{D}_2}}}(G_{ \mathcal{D}_1}({\tilde 
	x}_s))$. However, ${\mathbb{E}_{\mathcal{A},\mathcal{D}}}[ {\nabla {{\hat f}_{\cal 
			D}}({{\tilde 
			x}_s})} ] \ne \nabla f({{\tilde x}_s})$.  Then, we estimate the gradient of the  
$f(x_k)$ by 
\begin{align}\label{VRNonCS-SCSG:Definition:Estimate-f}
\nabla {{\tilde f}_k} = {(\partial {G_{{j_k}}}({x_k}))^\mathsf{T}}\nabla 
{F_{{i_k}}}({{\hat G}_k}) 
- {(\partial {G_{{j_k}}}({{\tilde x}_s}))^\mathsf{T}}\nabla 
{F_{{i_k}}}({G_{\mathcal{D}_1}}({{\tilde 
		x}_s})) + 
\nabla {{\hat f}_\mathcal{D}}({{\tilde x}_s}),
\end{align} 
where $i_k$ and $j_k$ are  
randomly selected from [n] at the $k$-th iteration for function $F$ and $G$, respectively. 
Furthermore, ${\mathbb{E}_{i_j,j_k\mathcal{A},\mathcal{D}}}[ {\nabla {{\tilde f}_k}} ] \ne \nabla 
f({x_k})$ as well. This gives us more discussion about the upper bound with respect to the 
estimated 
function and the gradient under the new random subset $\cal D$.
\subsection{Technical Tool}
For the subset $\mathcal{A} \subseteq [n]$, we present the following lemma that  the variance of a 
random variable decreases by a factor $|\cal A|$ if we choose
$|\cal A|$ independent element from $[n]$ and average them. The proof process is trivial and can 
be referred to Appendix. However, it present an important tool for analyzing the query complexity 
under the different size of the subset.
\begin{lemma}\label{VRNonCS-SCSG:Appendix:Lemma:Inequation-indepent}
	If $v_1,...,v_m\in \mathbb{R}^d$ satisfy 
	$\sum\nolimits_{i = 1}^m {{v_i}}  = \vec 0$, and $\cal A$ is a non-empty, 
	uniform 
	random subset of $[m]$, $A=|\cal A|$, then
	\begin{center}
		${\mathbb{E}_{\cal A}} {{{\left\| {\frac{1}{A}\sum\nolimits_{b \in 
							{\cal 
								A}} 
						{{v_b}} 
					} 
					\right\|}^2}}  \le \frac{{\mathbb{I}\left( {A < m} 
				\right)}}{A}\frac{1}{m}\sum\limits_{i = 1}^m {v_i^2}.$
	\end{center}
%
	Furthermore, if the elements in $\cal A$ are independent, then
	\begin{center}
		$	{\mathbb{E}_\mathcal{A}}{{{\left\| {\frac{1}{A}\sum\nolimits_{b \in 
							\mathcal{A}} {{v_b}} } 
					\right\|}^2}} = \frac{1}{{An}}\sum\limits_{i = 1}^n 
		{v_i^2}.$
	\end{center} 
%
\end{lemma}

Based on Lemma \ref{VRNonCS-SCSG:Appendix:Lemma:Inequation-indepent}, we can obtain the 
inequality with two-variables ${\mathcal{D}_1}$ and ${\mathcal{D}_2}$, which are used for the 
gradient of $f(x)$ with the partial 
gradient 
$\partial G(x)$.
\begin{lemma}\label{VRNonCS-SCSG:Appendix:Lemma:Inequation-indepent-double}
	If $w_1,...,w_n\in \mathbb{R}^{M\times N}$ and $v_1,...,v_n\in \mathbb{R}^M$ satisfy 
	$(\frac{1}{{{n}}} \sum\nolimits_{i \in [n]} {{w_{i}}}  )^\mathsf{T}( 
	\frac{1}{{{n}}}
	\sum\nolimits_{j 
		\in [n]} {{v_{j}}}  ) = \bar w^\mathsf{T}\bar v$, and $\mathcal{D} = \left[ 
	{{\mathcal{D}_1},{\mathcal{D}_2}} \right]$ is a non-empty, 
	uniform random subset consist of ${{\mathcal{D}_1}} $ and $ {{\mathcal{D}_2}} $, which are 
	independently and 
	uniformly selected from $[n]$, $D=|\mathcal{D}_1|=|\mathcal{D}_2|$, then
	{
	\begin{align*}
	{\mathbb{E}_\mathcal{D}}{\left\| {\frac{1}{{\left| {{\mathcal{D}_1}} \right|\left| 
					{{\mathcal{D}_2}} \right|}}\left( 
			{\sum\nolimits_{d_1 \in {\mathcal{D}_1}} {{w_{d_1}}} } \right)^\mathsf{T}\left( 
			{\sum\nolimits_{d_2 
					\in 
					{\mathcal{D}_2}} {{v_{d_2}}} } 
			\right) - \bar w\bar v} \right\|^2} =& {\mathbb{E}_\mathcal{D}}{\left\| 
		{\frac{1}{D^2}\left( {\sum\nolimits_{\left[ 
					{{d_1},{d_2}} 
					\right] \in 
					{\cal D}} {\left( {({w_{{d_1}}})^\mathsf{T}{v_{{d_2}}} - \bar w^\mathsf{T}\bar 
						v} \right)} 
			} 
			\right)} \right\|^2}\\
	\le&\frac{{\mathbb{I}\left({D^2<n^2} \right)}}{D^2}\frac{1}{{{n^2}}}\sum\limits_{i,j 
		= 1}^n 
	{{{\left\| {({w_i})^\mathsf{T}{v_j} - \bar w^\mathsf{T}\bar v} \right\|}^2}}.
	\end{align*}}
\end{lemma}
\subsection{Bounds analysis of the estimated function and the gradient} 
Here, we mainly give  different kinds of bounds for the proposed algorithm, such as  
${\mathbb{E}_{\mathcal{A},\mathcal{D}_1}}\| {\hat G}_k - 
G(x_k) \|^2$, 
$\mathbb{E}_{\mathcal{A,D}}\| E_{i_k,j_k}[ \nabla {\tilde f}_k ] - \nabla 
f(x_k) \|^2$ and $\mathbb{E}_{i_k,j_k,\cal A, \cal D}\| \nabla {\tilde f}_k - \nabla 
f( x_k 	) \|^2$. 
These bounds will be used to analyze the convergence rate and query complexity. We assume that 
these bound are all base on Assumption 
\ref{VRNonCS-SCSG:Assumption:G}-\ref{VRNonCS-SCSG:Assumption:Middle-Bound}. Parameters such as 
$B_G$, $B_F$, $L_G$, $L_F$ and $L_f$ in the bound are all from these Assumptions. We do not define 
the exact value of parameters such as $h$,  $A$  and $D$, which have great influence on the 
convergence and will be clearly defined in the query analysis. Our proposed bound are similar to 
that of \cite{lian2016finite}, \cite{liu2017duality} and \cite{liu2017variance}, but, the 
difference lies that there is an extra subset $\cal D$, which shows an interesting  phenomenon. 
That 
is when the subset  $\cal D$ is equal to the $[n]$, the corresponding bounds are the same as in 
\cite{lian2016finite}, \cite{liu2017duality} and \cite{liu2017variance}. However, it is the 
independent subset $\cal D$ that gives more general query complexity result for the problem 
(\ref{VRNonCS-SCSG:ProblemMainComposition}). The following bounds are all used for the composition 
problem for both convex and non-convex problem based on the Lemma 
\ref{VRNonCS-SCSG:Appendix:Lemma:Inequation-indepent} and Lemma 
\ref{VRNonCS-SCSG:Appendix:Lemma:Inequation-indepent-double}.  The more details of the proof can be 
referred to Appendix. For simplicity, we drop the 
superscript $i_k$, 
$j_k$, $\cal 
A$ and $\cal D$ for the expectation with  $\mathbb{E}$ in the proof.

\begin{lemma}\label{VRNonCS-SCSG:Lemma:Bound-estimate-G-2}
	Suppose Assumption 
	\ref{VRNonCS-SCSG:Assumption:G} and \ref{VRNonCS-SCSG:Assumption:Middle-Bound} holds, for 
	${{\hat G}_k}$ defined in (\ref{VRNonCS-SCSG:Definition:Estimate-G}) with $D =\left|{\cal 
		D}_1\right|$ 
	and  $A =\left|{\cal A}\right|$, 
	we have
	{
	\begin{align*}
	{\mathbb{E}_{\mathcal{A},\mathcal{D}_1}}{\| {{{\hat G}_k} - G({x_k})} \|^2} \le 
	4\left( {\frac{{\mathbb{I}\left( {A < n} \right)}}{A} + \frac{{\mathbb{I}\left( {D < n} 
				\right)}}{D}} 
	\right)B_G^2\mathbb{E}{\left\| {{x_k} - {{\tilde x}_s}} \right\|^2} + 
	2\frac{{\mathbb{I}\left( {D < 
				n} 
			\right)}}{D}{H_1}.
	\end{align*}}
\end{lemma}
\begin{lemma}\label{VRNonCS-SCSG:Lemma:Bound-estimate-G-3}
	Suppose 
	Assumption 
	\ref{VRNonCS-SCSG:Assumption:G}, \ref{VRNonCS-SCSG:Assumption:F}, 
	\ref{VRNonCS-SCSG:AssumptionIndependent} and
	\ref{VRNonCS-SCSG:Assumption:Middle-Bound} holds, for ${{\hat G}_k}$ defined in 
	(\ref{VRNonCS-SCSG:Definition:Estimate-G}) and ${\nabla {{\tilde f}_k}}$ defined in 
	(\ref{VRNonCS-SCSG:Definition:Estimate-f}) with $\mathcal{D} = \left[ 
	{{\mathcal{D}_1},{\mathcal{D}_2}} \right]$ and $D =\left|{\cal D}_1\right|=\left|{\cal 
		D}_2\right|$, we have
	{
	\begin{align*}
	{\mathbb{E}_{\mathcal{A,D}}}{\| {{\mathbb{E}_{i_k,j_k}}[ {\nabla {{\tilde f}_k}} 
			] - \nabla 
			f({x_k})} 
		\|^2}
	\le& 4B_G^4L_F^2\left( {4\frac{{\mathbb{I}\left( {A < n} \right)}}{A} + 
		4\frac{{\mathbb{I}\left( {D < n} 
				\right)}}{D}} \right)\mathbb{E}{\left\| {{x_k} - {{\tilde x}_s}} \right\|^2}\\& + 
	16B_G^2L_F^2\frac{{\mathbb{I}(D 
			< n)}}{D}{H_1} + 4\frac{{\mathbb{I}(D^2 < n^2)}}{{D^2}}{H_2}.
	\end{align*}}
\end{lemma}

\begin{lemma}\label{VRNonCS-SCSG:Lemma:Bound-estimate-G-4}
	Suppose 
	Assumption \ref{VRNonCS-SCSG:Assumption:G}-\ref{VRNonCS-SCSG:Assumption:Middle-Bound} holds, 
	for 
	${{\hat 
			G}_k}$ defined in (\ref{VRNonCS-SCSG:Definition:Estimate-G}) and ${\nabla {{\tilde 
				f}_k}}$ defined in (\ref{VRNonCS-SCSG:Definition:Estimate-f}) with $\mathcal{D} = 
	\left[ 
	{{\mathcal{D}_1},{\mathcal{D}_2}} \right]$ and $D =\left|{\cal 
		D}_1\right|=\left|{\cal 
		D}_2\right|$, we have
	{
	\begin{align*}
	\mathbb{E}_{i_k,j_k,\cal A, \cal D}{\| {\nabla {{\tilde f}_k} - \nabla f\left( {{x_k}} 
			\right)} \|^2}
	\le& 5B_G^4L_F^2\left( 
	{\frac{L_f^2}{B_G^4L_F^2} + 
		4\frac{{\mathbb{I}\left( {A < n} \right)}}{A} + 4\frac{{\mathbb{I}\left( {D < n} 
				\right)}}{D}} \right)\mathbb{E}{\left\| {{x_k} - {{\tilde x}_s}} \right\|^2}\\& + 
	20B_G^2L_F^2\frac{{\mathbb{I}(D < 
			n)}}{D}{H_1} + 5\frac{{\mathbb{I}(D^2 < n^2)}}{{D^2}}{H_2}.
	\end{align*}}
\end{lemma}
As can be seen from the above results directly, when $A$ and $D$ increase, the upper bounds are 
more approximating the related bounds as in \cite{lian2016finite,liu2017duality,liu2017variance}. 
Though there are extra terms with respect to $A$ and $D$, it gives us another direction for 
analyzing the convergence rate and query complexity. As the convergence rate not only depends on 
the 
convergence sequence, but also the terms including the event function $\mathbb{I}$. Thus, we can 
obtain the lower bound range of $A$ and $D$ that is related to $\varepsilon$.
The result in Lemma \ref{VRNonCS-SCSG:Lemma:Bound-estimate-G-3} and 
\ref{VRNonCS-SCSG:Lemma:Bound-estimate-G-4} are similar except the extra term 
$L_f^2\mathbb{E}||x_k-\tilde{x}_s||^2$. This is due to the fact that the order of the expectation 
is different. This difference derives from the proof process by using the smoothness of the 
function $f(x)$ and the update of $x$ in Algorithm \ref{VRNonCS-SCSG:AlgorithmI}. Furthermore, 
these two lemmas can be both applied to analyze the convergence rate and query complexity of  the 
convex and non-convex composition problem.
\section{Stochastic Composition via SCSG method for the Convex Composition problem} 
\label{VRNonCS-SCSG:Section:SCSG-convex}
In this section, we analyze the proposed method for the convex composition problem. We first 
present the convergence of the proposed algorithm, and then give the query complexity. Thought the 
proof  is similar to that of \cite{lian2016finite} and \cite{xiao2014proximal}, we present a more 
clear and simple process as there is an extra term deriving from the subset $\cal D$. In order to 
ensure the convergence of the proposed algorithm, we obtain the desired parameters' setting, such 
as $A$, $D$, $K$, 
$\eta$ and $h$. Based on the setting, we can obtain the corresponding query complexity, which is 
better than or equal to the SVRG-based method in \cite{lian2016finite} and \cite{liu2017duality}. 
This is in fact that the event function $\mathbb{I}$ has the influence on the size of $A$ and $D$.
\subsection{Convergence  analysis for the convex problem}
Based on the strong convex and smoothness of the function of $f(x)$, we provide the convergence 
sequence, in which the parameters are not defined. But the sequences motivate us to consider the 
parameters' setting such that lead to the desired convergence rate.
\begin{theorem}\label{VRNonCS-SCSG:Lemma:NewSequence}
	Suppose Assumption \ref{VRNonCS-SCSG:Assumption:f-strong}-
	\ref{VRNonCS-SCSG:Assumption:Middle-Bound} holds, in Algorithm \ref{VRNonCS-SCSG:AlgorithmI}, 
	let $h>0, \eta>0$, $A=|\cal A|$,  $D=|\mathcal{D}_1|=|\mathcal{D}_2|$, $K$ is the number of the 
	inner iteration, $x^*$ is the optimal point, we have
	\begin{center}
		${\rho _1}\mathbb{E}{\left\| {{{\tilde x}_{s + 1}} - {x^*}} \right\|^2} \le \left( 
		{\frac{1}{K} 
			+ {\rho 
				_2}} \right)\mathbb{E}{\left\| {{{\tilde x}_s} - {x^*}} \right\|^2} +\rho_3,$
	\end{center}
%
	where
	{
	\begin{align}
	\label{VRNonCS-SCSG:Lemma:NewSequence-V}
	V =& B_G^4L_F^2\left( {4\frac{{\mathbb{I}\left( {A < n} \right)}}{A} + 4\frac{{\mathbb{I}\left( 
				{D < n} 
				\right)}}{D}} \right),\\
	\label{VRNonCS-SCSG:Lemma:NewSequence-V1}
	{V_1} =& 20B_G^2L_F^2\frac{{\mathbb{I}(D < n)}}{D}{H_1} + 5\frac{{\mathbb{I}({D^2} < 
			{n^2})}}{{{D^2}}}{H_2},\\
	\label{VRNonCS-SCSG:Lemma:NewSequence-r1}
	{\rho _1} =& \left( {2\mu  - h - 4V\frac{1}{h} - \left( {12L_f^2 + 10V} \right)\eta } 
	\right)\eta, \\
	\label{VRNonCS-SCSG:Lemma:NewSequence-r2}
	{\rho _2} =& 2\left( {2V\frac{1}{h} + 5\left( {L_f^2 + V} \right)\eta } \right)\eta,\\
	\label{VRNonCS-SCSG:Lemma:NewSequence-r3}
	{\rho _3} =& \frac{1}{h}\eta {\frac{4}{5}{V_1}} 
	+ 2{\eta ^2}{V_1}.
	\end{align}}
\end{theorem}
We do not give the convergence form for the update of iteration as the we do no sure whether 
$\rho_1$ 
is positive or not. Based on above equality in Lemma \ref{VRNonCS-SCSG:Lemma:NewSequence}, we 
assume that $\rho_1>0$ in 
(\ref{VRNonCS-SCSG:Lemma:NewSequence-r1}), 
then we can obtain
{
\begin{align}\label{VRNonCS-SCSG:Lemma:NewSequence-convergence}
\mathbb{E}{\left\| {{{\tilde x}_S} - {x^*}} \right\|^2} \le& {\rho ^S}\mathbb{E}{\left\| 
	{{{\tilde x}_0} - 
		{x^*}} \right\|^2} + \frac{{{\rho _3}}}{{{\rho _1}}}\sum\limits_{s = 0}^S {{\rho ^s}}
\nonumber\\  \le &
{\rho ^S}\mathbb{E}{\left\| {{{\tilde x}_0} - {x^*}} \right\|^2} + \frac{{{\rho 
			_3}}}{{{\rho 
			_1}}}\frac{1-\rho^S}{{1 - \rho }},
\end{align}}
where $\rho=(\frac{1}{K}+\rho_2)/\rho_1$, $\rho_2$ and $\rho_3$ defined in 
(\ref{VRNonCS-SCSG:Lemma:NewSequence-r2}) and (\ref{VRNonCS-SCSG:Lemma:NewSequence-r3}), the last 
inequality is based on the formula of geometric progression. Thus, if the $\tilde{x}_S$ converge 
to the optimal point $x^*$, we need to require that $\rho<1$ and the second term ${\rho _3}( 
{1 - {\rho ^S}} )/( {{\rho _1}( {1 - \rho } )} )$ is less than 
$\varepsilon/2$. Actually, if $D=n$, the second term is equal to zero satisfying the requirement 
directly, which is similar to the 
convergence results in \cite{lian2016finite} and \cite{liu2017duality}.
\subsection{Query complexity analysis for the convex problem}
Based on the above result in (\ref{VRNonCS-SCSG:Lemma:NewSequence-convergence}), we analyze the 
query complexity. Furthermore, we also present the parameters' setting, and then derived the query 
complexity, in which the details information can be referred to the Appendix.
\begin{corollary}\label{VRNonCS-SCSG:Lemma:paraeters setting}
	Suppose Assumption \ref{VRNonCS-SCSG:Assumption:f-strong}-
	\ref{VRNonCS-SCSG:Assumption:Middle-Bound} holds, in Algorithm \ref{VRNonCS-SCSG:AlgorithmI}, 
	let $h=\mu$, the step size is $\eta  \le {\mu 
	}/({{135L_f^2}})$, the subset size of $\cal A$ is $ A = \min \{ {n,128B_G^4L_F^2/{{{\mu 
					^2}}}} \}$, the subset size of $\mathcal{D}_1$ and $\mathcal{D}_2$ are both $D 
					= \min \left\{ 
	{n,5\left( {16B_G^4L_F^2{H_1} + 
			4{H_2}} \right)/({{4\varepsilon {\mu ^2}}}}) \right\}$,  the number 
	of the inner iteration is $K \ge 540L_f^2/\mu ^2$, the number of outer iteration is 
	$S \ge 1/( \log ( 1/\rho  ) )\log ( 
	2E\| {\tilde x}_0 - x^* \|^2/\varepsilon )$.	the query complexity is 
{\small
	\begin{align*}
	{\left( {D + KA} \right)S} = 
\mathcal{O}\left( 
{\left( {\min \left\{ {n,\frac{1}{{{\varepsilon\mu ^2}}}} \right\} + \frac{{L_f^2}}{{{\mu 
					^2}}}\min 
		\left\{ {n,\frac{1}{{{\mu ^2}}}} \right\}} \right)\log \left( {1/\varepsilon } 
	\right)} 
\right).
	\end{align*}
}
\end{corollary}
As can be seen from the above result, Corollary \ref{VRNonCS-SCSG:Lemma:paraeters setting} present 
the 
general query complexity under different parameters. Comparing $n$ with corresponding parameters, 
we analyze the query complexity separately. We remove the parameters such as $B_G^2$, $L_F^2$, 
$H_1$ and $H_2$, and analyze the size with the order of $1/\mu$.  Though the comparison is not 
exactly correct, we present the results to illustrate the corresponding different  algorithms. We 
can directly obtain that $1/\mu^2<1/(\varepsilon\mu^2)$. We consider three situations comparing 
with $n$, that is to present the value of the $min$ function,

\begin{itemize}
	\item $\frac{1}{{{\mu ^2}}} \le \frac{1}{{\varepsilon {\mu ^2}}} \le n$. When n is large enough 
	such that we can obtain the  query complexity 
	is ${\cal O}( {( {1/( {\varepsilon {\mu ^2}} ) + {}L_f^2/{\mu ^4}} 
		)\log ( {1/\varepsilon } )} )$. This result avoids the situation that computing the full 
	gradient 
	of $f(x)$ and the full function $G(x)$ for the large-scale number of n. What's more, this 
	result is 
	better than Compositional-SVRG \cite{lian2016finite} and \cite{liu2017duality}.
	\item $\frac{1}{{{\mu ^2}}} \le n \le \frac{1}{{\varepsilon {\mu ^2}}}$. When n is smaller than 
	$1/(\varepsilon\mu^2)$,   the query complexity 
	becomes ${\cal O}( {( {n + {}L_f^2/{\mu ^4}} )\log ( {1/\varepsilon } 
		)} )$, which is the same as  Compositional-SVRG \cite{lian2016finite}. That is we need to 
	compute the full gradient of $\nabla f(\tilde{x}_s)$ as in 
	(\ref{VRNonCS-SCSG:Definition:Estimate-f-svrg}). The estimation of inner function $G(x)$ is the 
	same as in \cite{lian2016finite}.
	\item $n \le \frac{1}{{{\mu ^2}}} \le \frac{1}{{\varepsilon {\mu ^2}}}$. When $n$ is small, the 
	query complexity 
	becomes ${\cal O}( {( {n + L_f^2n/{\mu ^2}} )\log ( {1/\varepsilon } 
		)} )$. The result has the similar form to SVRG \cite{johnson2013accelerating}. This also 
		gives
	us the intuition that the inner function should be computed directly rather than estimated.
\end{itemize}

\section{Stochastic Composition via SCSG method for the Non-convex composition problem} 
\label{VRNonCS-SCSG:Section:SCSG-Non-convex}
In this section, we give the analysis of the convergence 
analysis and the query complexity under the proposed algorithm for the non-convex composition. We 
first present the new reformed sequence with respect to $E[f({x_k})] + {c_k}E{\left\| {{x_k} - 
		{{\tilde x}_s}} 	\right\|^2}$, in which the parameters are not well defined. Then, we 
		sum-up 
these sequence 	based on the SVRG-based on the framework, in which there is a snapshot point 
${\tilde 
	x}_s$ in each epoch. The last not least, we present the query complexity analysis and derive 
	the 
optimal parameters' setting such that improve the query complexity.
\subsection{Convergence analysis for the non-convex problem}
We first present the new form sequence under the Lyapunov function based on the smoothness of 
$f(x)$ 
and the update of $x$. The new parameters such as $c_k$, $u_k$ and $J_k$ will be used to form 
sequence such that we can obtain the convergence sequence. 
\begin{lemma}\label{VRNonCS-SCSG:Non:Lemma:NewSequence}
	Suppose Assumption \ref{VRNonCS-SCSG:Assumption:G}-
	\ref{VRNonCS-SCSG:Assumption:Middle-Bound} holds, in Algorithm \ref{VRNonCS-SCSG:AlgorithmI}, 
	we can 
	obtain the following new sequence with respect to $f(x_k)$ and $||x_k-\tilde{x}_s||^2$, let 
	$h>0,\eta>0$, $A =\left|{\cal A}\right|$ and $D =\left|{\cal D}_1\right|=\left|{\cal 
		D}_2\right|$,  we have
	\begin{align*}
	\mathbb{E}[f({x_{k + 1}})] + {c_{k + 1}}\mathbb{E}{\left\| {{x_{k + 1}} - {{\tilde x}_s}} 
		\right\|^2}
	\le \mathbb{E}[f({x_k})] + {c_k}\mathbb{E}{\left\| {{x_k} - {{\tilde x}_s}} \right\|^2} - 
	{u_k}{\left\| 
		{\nabla f({x_k})} \right\|^2} +J_k,
	\end{align*}	
	where
	{
	\begin{align}
	\label{VRNonCS-SCSG:Non:Lemma:NewSequence-W}			
	W =& B_G^4L_F^2\left( {4\frac{{\mathbb{I}\left( {A < n} \right)}}{A} + 4\frac{{\mathbb{I}\left( 
				{D < n} 
				\right)}}{D}} \right),\\
	\label{VRNonCS-SCSG:Non:Lemma:NewSequence-c_k}
	{c_k} =& {c_{k + 1}}\left( {1 + \left( {\frac{2}{h} + 4hW} \right)\eta  + 10\left( 
		{L_f^2 + W} \right){\eta ^2}} \right) + 2W\eta  + 5(L_f^2+W){L_f}{\eta ^2},\\
	\label{VRNonCS-SCSG:Non:Lemma:NewSequence-u_k}
	{u_k} =& \left( {\left( 
		{\frac{1}{2} - h{c_{k + 1}}} \right)\eta  - \left( {{L_f} + 2{c_{k 
					+ 1}}} \right){\eta ^2}} \right),\\
	\label{VRNonCS-SCSG:Non:Lemma:NewSequence-W_1}
	{W_1} =& 20B_G^2L_F^2\frac{{\mathbb{I}(D < n)}}{D}{H_1} + 5\frac{{\mathbb{I}(D^2 < 
			n^2)}}{{D^2}}{H_2},\\
	\label{VRNonCS-SCSG:Non:Lemma:NewSequence-J_k}
	J_k=& \left( {\frac{1}{2} + h{c_{k + 1}}} \right){\frac{4}{5}W_1}\eta  + 
	\left( {{L_f} + 2{c_{k + 1}}} \right){W_1}{\eta ^2}.
	\end{align}
}
\end{lemma}

Based on the above inequality with respect to the sequence $\mathbb{E}[f({x_k})] + 
{c_k}\mathbb{E}{\left\| {{x_k} - {{\tilde x}_s}} \right\|^2}$ and Algorithm 
\ref{VRNonCS-SCSG:AlgorithmI}, we can obtain the convergence form in which the parameters are not 
clearly defined.
\begin{theorem}\label{VRNonCS-SCSG:Non:Lemma:Convergence form}
	Suppose Assumption \ref{VRNonCS-SCSG:Assumption:G}-
	\ref{VRNonCS-SCSG:Assumption:Middle-Bound} holds, in Algorithm \ref{VRNonCS-SCSG:AlgorithmI}, 
	we can 
	obtain the following new sequence with respect to $f(x_k)$ and $||x_k-\tilde{x}_s||^2$.  $K$ is 
	the number of inner iterations, $S$ is the  number of 
	inner iterations, we have
	\begin{align*}
	u_0\mathbb{E}[{\| {\nabla f({\hat x_k^s })} \|^2}]  \le 
	\frac{{f({x_0}) - f({x^*})}}{{KS}}+J_0,
	\end{align*}
	where  $\hat x_k^s$ is the output point, $J_0$ and $u_0$ are defined in 
	(\ref{VRNonCS-SCSG:Non:Lemma:NewSequence-J_k}) and 
	(\ref{VRNonCS-SCSG:Non:Lemma:NewSequence-u_k}).
\end{theorem}
\subsection{Query complexity analysis for the non-convex problem}
Consider the convergence form above,  we actually can't obtain the convergence rate if the 
parameter $u_k$ in (\ref{VRNonCS-SCSG:Non:Lemma:NewSequence-u_k}) is negative. Furthermore, there 
is extra term $J_0$ derived from the subset $\cal D$. We need to consider the size of the subset 
$\cal D$ such that we can keep the $J_0$ under our desired  degree of accuracy $\epsilon$. What's 
more, the parameter $c_k$ in (\ref{VRNonCS-SCSG:Non:Lemma:NewSequence-c_k}) is not a constant, 
which has a relationship with $K$ and $\eta$. Based on these influence element, we consider the 
parameters' setting and give the query complexity.
\begin{corollary}\label{VRNonCS-SCSG:Non:Lemma:paraeters setting}
	Suppose Assumption \ref{VRNonCS-SCSG:Assumption:G}-
	\ref{VRNonCS-SCSG:Assumption:Middle-Bound} holds, in Algorithm \ref{VRNonCS-SCSG:AlgorithmI}, 
	for the step $\eta>0$, by setting $h =\sqrt {1/\eta}$,  the set-size of the subset 
	$\mathcal{D}_1$ and 
	$\mathcal{D}_2$ are $ D = \min \left\{ {n,\mathcal{O}(1/\varepsilon )} \right\}$, the set-size 
	of $\cal A$ is $A = \min \left\{ {n,\mathcal{O}\left( {1/\eta } \right)} \right\}$, the number 
	of 
	inner iteration is $K \le 
	\mathcal{O}\left( {1/{\eta ^{3/2}}} \right)$, the total number of iteration is $T = 
	\mathcal{O}\left( {1/\left( 
		{\varepsilon \eta } \right)} \right)$, then we can 
	obtain	$\mathbb{E}[\| \nabla f(\hat x_k^s ) \|^2]  \le 
	\varepsilon$.
\end{corollary}

The above corollary gives the parameters' setting except the step $\eta$, note that the outer 
number of iteration $S$ has the relationship with $T$ and $K$, that is $T=SK$. Here, we present the 
optimal step $\eta$ such that we reach the improved the query complexity.
\begin{corollary}
	Suppose Assumption \ref{VRNonCS-SCSG:Assumption:G}-
	\ref{VRNonCS-SCSG:Assumption:Middle-Bound} holds, in Algorithm \ref{VRNonCS-SCSG:AlgorithmI}, 
	the step size is  $\eta  = \min \{ 1/n^{2/5},\varepsilon ^{2/5} \}$, then 
	the query complexity is
	\begin{center}
		$\mathcal{O}\left( \min \left\{ 
		{\frac{1}{{{\varepsilon 
						^{9/5}}}},\frac{{{n^{4/5}}}}{\varepsilon }} 
		\right\} \right),$
	\end{center}
%
\end{corollary}
\begin{proof}
	Based on the parameters' setting, that is $ D = \min \left\{ {n,\mathcal{O}(1/\varepsilon )} 
	\right\}$, $A = \min \left\{ {n,\mathcal{O}\left( {1/\eta } \right)} \right\}$, $K \le 
	\mathcal{O}\left( {1/{\eta ^{3/2}}} \right)$, and  $T = 
	\mathcal{O}\left( {1/\left( 
		{\varepsilon \eta } \right)} \right)$, we have,
	{
	\begin{align*}	
	\mathcal{O}\left( {\frac{T}{K}\left( {D + KA} \right)} \right) =& \mathcal{O}\left( 
	{\frac{1}{{\varepsilon \eta 
		}}\left( {\frac{D }{K} + A} \right)} \right)\\
	=& \mathcal{O}\left( {\frac{1}{{\varepsilon \eta }}\left( {\min \left\{ {n,\frac{1}{\varepsilon 
			}} 
			\right\} {{\eta ^{3/2}}} + \frac{1}{\eta }} \right)} \right)\\
	=& \mathcal{O}\left( {\frac{1}{\varepsilon }\left( {\min \left\{ {n,\frac{1}{\varepsilon }} 
			\right\}{{\eta ^{1/2}}}  + \frac{1}{{{\eta ^2}}}} \right)} \right)\\
	\ge& \mathcal{O}\left( \min \left\{ {\frac{1}{{{\varepsilon 
					^{9/5}}}},\frac{{{n^{4/5}}}}{\varepsilon }} 
	\right\} \right),
	\end{align*}
	}
	where the optimal $\eta  = \min \left\{ {1/{n^{2/5}},{\varepsilon ^{2/5}}} \right\}$.
\end{proof}
As can be sen from the above result,  we can see that when $n$ is large enough the query complexity 
become $ \mathcal{O}(1/\varepsilon^{9/5})$, that is the gradient and the inner function are 
estimated rather than computed the full value directly. The corresponding is better than the 
accelerated method in \cite{wang2016accelerating}, in which the query complexity does not depend on 
$n$.
Furthermore,  when $n\le 1/\varepsilon$, the 
query complexity is $\mathcal{O}(n^{4/5}/\varepsilon)$, which is the same as in 
\cite{liu2017variance} for the case of the problem in (\ref{VRNonCS-SCSG:ProblemMainComposition}).
\section{Mini-batch version of SC-SCSG for the composition problem} 
\label{VRNonCS-SCSG:Section:SCSG-minibatch}
\begin{algorithm}[t]
	\caption{Mini-batch version of SC-SCSG for the  
	composition problem}
	\label{VRNonCS-SCSG:AlgorithmII}
	\begin{algorithmic}
		\Require $K$, $S$, $\eta$ (learning rate), $\tilde{x}_0$ and $\mathcal{D} = \left[ 
		{{\mathcal{D}_1},{\mathcal{D}_2}} \right]$
		\For{$s =0,1, 2,\cdots,S-1$} 
		\State Sample from $[n]$ for D times to form mini-batch $\mathcal{D}_1$ 
		\State Sample from $[n]$ for D times to form mini-batch $\mathcal{D}_2$ 		
		\State $\nabla {\hat f_{ \mathcal{D}}}({{\tilde x}_s}) = {(\partial G_{ 
				\mathcal{D}_1}({{\tilde 
					x}_s}))^\mathsf{T}}{\nabla {F_{{ \mathcal{D}_2}}}(G_{ 
				\mathcal{D}_1}({{\tilde 
					x}_s}))}  
		$\Comment{D Queries}
		\State $x_0=\tilde{x}_s$
		\For{$k =0,1,2,\cdots,K-1$}
		\State
		Sample from $[m]$  to form mini-batch $\mathcal{A}$ 
		\State
		${{\hat G}_k} = {G_{{{\cal A}}}}({x_k}) - {G_{{{\cal A}}}}({{\tilde x}_s}) + 
		{G_{\mathcal{D}_1}}({{\tilde 
				x}_s})$\Comment{A Queries}
		\State ${\Lambda_0}=0$	
		\For{t=1,...,b}	
		\State Uniformly and randomly pick $i_k$ and $j_k$ from $[n]$				
		\State
		Compute the estimated gradient $\nabla {{\tilde f}_k}$ from 
		(\ref{VRNonCS-SCSG:Definition:Estimate-f})
		\Comment{4 Queries}
		\State ${\Lambda _{t+1}}={\Lambda _t}+\nabla {{\tilde f}_k}$
		\EndFor
		\State ${\Lambda}={\Lambda _b}/b$
		\State
		${x_{k+1}}= {x_k} - {\eta} {\Lambda} $
		\EndFor
		\State Update $\tilde{x}_{s+1}=x_K $
		\EndFor \\	
		\textbf{Output:}  $\hat x_k^s$ is uniformly and randomly chosen from  $s\in\{0,...,S-1\}$ 
		and $k\in \{0,..,K-1\}$.
	\end{algorithmic}
\end{algorithm}
In this section, we present the mini-batch version of the proposed method in Algorithm 
\ref{VRNonCS-SCSG:AlgorithmII}. The difference with Algorithm \ref{VRNonCS-SCSG:AlgorithmI} is the 
computation of the gradient of the $f(x)$. Furthermore, the convergence proof with the related 
upper bounds are almost the same except the following lemma. By using the Lemma 
\ref{VRNonCS-SCSG:Appendix:Lemma:Inequation-indepent}, we derive the similar bound but the first 
term is reduced by a factor of $b$, where $b$ is the number of the mini-batch. Note that here the 
element in the mini-batch are independent, we can obtain the result directly. The details can be 
referred to Appendix.
\begin{lemma}\label{VRNonCS-SCSG-miniBatch:Lemma:Bound-estimate-G-4}
	Suppose 
	Assumption \ref{VRNonCS-SCSG:Assumption:G}-\ref{VRNonCS-SCSG:Assumption:Middle-Bound} holds, 
	for 
	${{\hat 
			G}_k}$ defined in (\ref{VRNonCS-SCSG:Definition:Estimate-G}) and $\Lambda $ defined in 
	Algorithm \ref{VRNonCS-SCSG:AlgorithmII} with $\mathcal{D} = 
	\left[ 
	{{\mathcal{D}_1},{\mathcal{D}_2}} \right]$ and $D =\left|{\cal 
		D}_1\right|=\left|{\cal 
		D}_2\right|$, we have
	\begin{align*}
	&\mathbb{E}_{i_k,j_k,\cal A, \cal D}{\| {\Lambda  - \nabla f\left( {{x_k}} 
			\right)} \|^2}\\
	\le&5B_G^4L_F^2\left( 
	{\frac{L_f^2}{bB_G^4L_F^2} + 
		4\frac{{\mathbb{I}\left( {A < n} \right)}}{A} + 4\frac{{\mathbb{I}\left( {D < n} 
				\right)}}{D}} \right)\mathbb{E}{\left\| {{x_k} - {{\tilde x}_s}} \right\|^2} + 
	20B_G^2L_F^2\frac{{\mathbb{I}(D < 
			n)}}{D}{H_1} + 5\frac{{\mathbb{I}(D^2 < n^2)}}{{D^2}}{H_2},
	\end{align*}
\end{lemma}
Based on the above lemma, we can obtain the query complexity for both convex and non-convex 
problem. As the process of the proof are similar to that of Corollary 
\ref{VRNonCS-SCSG:Lemma:paraeters setting} and Corollary 
\ref{VRNonCS-SCSG:Non:Lemma:paraeters setting}, we give the following result directly. The 
difference of the parameters' setting are $K$, and $\eta$ due to the fact of the mini-batch.
\begin{corollary}\label{VRNonCS-SCSG-miniBatch:Lemma:paraeters setting}
	Suppose Assumption \ref{VRNonCS-SCSG:Assumption:f-strong}-
	\ref{VRNonCS-SCSG:Assumption:Middle-Bound} holds, in Algorithm \ref{VRNonCS-SCSG:AlgorithmII}, 
	for convex problem,	let $h=\mu$, the step size is $\eta  \le {b\mu 
	}/({{135L_f^2}})$, the subset size of $\cal A$ is $ A = \min \{ {n,128B_G^4L_F^2/{{{\mu 
					^2}}}} \}$, the subset size of $\mathcal{D}_1$ and $\mathcal{D}_2$ are both $D 
	= \min \{ 
	n,5( 16B_G^4L_F^2{H_1} + 
	4{H_2} )/(4\varepsilon \mu ^2) \}$,  the number 
	of the inner iteration is $K \ge 540L_f^2/(b\mu ^2)$, the number of outer iteration is 
	$S \ge 1/( \log ( 1/\rho  ) )\log ( 
	2E\| {\tilde x}_0 - x^* \|^2/\varepsilon )$.	The query complexity is 
	\begin{center}
		${\left( {D + KA} \right)S} = 
		\mathcal{O}\left( 
		{\left( {\min \left\{ {n,\frac{1}{{{\varepsilon\mu ^2}}}} \right\} + \frac{{L_f^2}}{{{b\mu 
							^2}}}\min 
				\left\{ {n,\frac{1}{{{\mu ^2}}}} \right\}} \right)\log \left( {1/\varepsilon } 
			\right)} 
		\right).$
	\end{center}
%
\end{corollary}
\begin{corollary}\label{VRNonCS-SCSG-miniBatch-non:query complexity}
	Suppose Assumption \ref{VRNonCS-SCSG:Assumption:G}-
\ref{VRNonCS-SCSG:Assumption:Middle-Bound} holds, in Algorithm \ref{VRNonCS-SCSG:AlgorithmII}, 
Let $h =\sqrt {b/\eta}$,  	the step size is  $\eta  = {b^{3/5}}\min \{ 1/n^{2/5},\varepsilon 
^{2/5} 
\}$, the set-size 
of $\cal A$ is $A = \min \left\{ {n,\mathcal{O}\left( {b/\eta } \right)} \right\}$, the 
set-size of the subset 
$\mathcal{D}_1$ and 
$\mathcal{D}_2$ are $ D = \min \left\{ {n,\mathcal{O}(1/\varepsilon )} \right\}$,  the number 
of 
inner iteration is $K \le 
\mathcal{O}\left( {b^{1/2}/({\eta ^{3/2}})} \right)$, the total number of iteration is $T = 
\mathcal{O}\left( {1/\left( 
	{\varepsilon \eta } \right)} \right)$, in order to obtain 
obtain	$\mathbb{E}[\| \nabla f(\hat x_k^s ) \|^2]  \le 
\varepsilon$.The query complexity is 
\begin{center}
	$\frac{1}{{{b^{1/5}}}}\mathcal{O}\left( {\min \left\{ {\frac{1}{{{\varepsilon 
						^{9/5}}}},\frac{{{n^{4/5}}}}{\varepsilon }} \right\}} \right)$
\end{center}
%
\end{corollary}
From the above result of the query complexity of the convex and non-convex problem, we can see that 
both of their step size $\eta$ and the  number of inner iteration $K$ increase. These two 
key parameters lead to the improved the query complexity of both convex and non-convex problem. 
\section{Conclusion}\label{VRNonCS-SCSG:Section:Conclusion}
In this paper, we propose the variance reduction based method for the convex and non-convex 
composition problem. We apply the stochastically 
controlled stochastic gradient to estimate  inner function $G(x)$ and  the gradient of 
$f(x)$. The query complexity of our proposed algorithm is better than or equal to the current 
methods on both  convex and non-convex function. Furthermore, we also present the corresponding 
mini-batch version of the proposed method, in which the query complexities are improved as well. In 
the future, we can consider the non-smooth function of the composition problem with the method of 
the 
stochastically controlled stochastic gradient. 

{\small
	\bibliographystyle{unsrt}
	\bibliography{SCSG_SCVR_Bib}
}
\appendix
\section{Technical Tool}
\begin{lemma2}\ref{VRNonCS-SCSG:Appendix:Lemma:Inequation-indepent}. 
	If $v_1,...,v_n\in \mathbb{R}^M$ satisfy 
	$\sum\nolimits_{i = 1}^n {{v_i}}  = \vec 0$, and $\cal A$ is a non-empty, 
	uniform 
	random subset of $[m]$, then
	\begin{align*}
	{\mathbb{E}_{\cal A}} {{{\left\| {\frac{1}{A}\sum\nolimits_{b \in 
						{\cal 
							A}} 
					{{v_b}} 
				} 
				\right\|}^2}}  \le \frac{{\mathbb{I}\left( {A < n} 
			\right)}}{A}\frac{1}{n}\sum\limits_{i = 1}^n {v_i^2}.
	\end{align*}
	Furthermore, if the elements in $\cal A$ are independent, then 
	\begin{align*}
	{\mathbb{E}_\mathcal{A}}{{{\left\| {\frac{1}{A}\sum\nolimits_{b \in 
						\mathcal{A}} {{v_b}} } 
				\right\|}^2}} = \frac{1}{{An}}\sum\limits_{i = 1}^n 
	{v_i^2}. 
	\end{align*}
\end{lemma2}
\begin{proof} Based on the  $\sum\nolimits_{i = 1}^m {{v_i}}  = \vec 0$, and 
	permutation and combination, 
	For the case that $\cal A$ is a non-empty, 
	uniformly random subset of $[m]$, we 
	have 
	\begin{align*}
	&{\mathbb{E}_{\cal A}} {{{\left\| {\sum\nolimits_{b \in \mathcal{A}} 
					{{v_b}} } 	\right\|}^2}} \\ =& {\mathbb{E}_{\cal 
			A}}\left[ {\sum\nolimits_{b 
			\in {\cal A}} 	{{{\left\| 	{{v_b}} \right\|}^2}} } \right] 
	+ 
	\frac{1}{{C_n^A}}\sum\limits_{i 	\in [n]} 
	{\left\langle {{v_i},\frac{{C_{n - 1}^{A - 1}\left( {A - 1} 
					\right)}}{{n - 
					1}}\sum\limits_{i \ne j} {{v_j}} } \right\rangle 
	} 
	\\
	=& A\frac{1}{n}\sum\nolimits_{i = 1}^n {v_i^2}  + \frac{{A\left( {A - 
				1} 
			\right)}}{{n\left( {n - 1} \right)}}\sum\nolimits_{i \in 
		[n]} 
	{\left\langle 
		{{v_i},\sum\nolimits_{i \ne j} {{v_j}} } \right\rangle } \\
	=& A\frac{1}{n}\sum\nolimits_{i = 1}^n {v_i^2}  + \frac{{A\left( {A - 
				1} 
			\right)}}{{n\left( {n - 1} \right)}}\sum\nolimits_{i \in 
		[n]} 
	{\left\langle 
		{{v_i}, - {v_i}} \right\rangle }\\ 
	=& \frac{{A\left( {n - A} \right)}}{{\left( {n - 1} 
			\right)}}\frac{1}{n}\sum\nolimits_{i = 1}^n {v_i^2}\\
	\le& A\mathbb{I}\left( {A < n} \right)\frac{1}{n}\sum\nolimits_{i = 
		1}^n 
	{v_i^2}. 
	\end{align*}
	Thus, we have
	\begin{align*}
	&{\mathbb{E}_{\cal A}} {{{\left\| \frac{1}{A}{\sum\nolimits_{b \in \mathcal{A}} 
					{{v_b}} } 	\right\|}^2}}=\frac{1}{A^2}{\mathbb{E}_{\cal A}} {{{\left\| 
				{\sum\nolimits_{b \in \mathcal{A}} 
					{{v_b}} } 	\right\|}^2}}\le\frac{{\mathbb{I}\left( {A < n} 
			\right)}}{A}\frac{1}{n}\sum\limits_{i = 1}^n {v_i^2}.
	\end{align*}
	
	For the case that the element in $\cal A$ is 
	randomly and independently selected from $[m]$, we 
	have

	\begin{align}
	{\mathbb{E}_\mathcal{A}} {{{\left\| {\sum\nolimits_{b \in 
						\mathcal{A}} {{v_b}} } 
				\right\|}^2}}  =& {\mathbb{E}_\mathcal{A}}\left[ 
	{\sum\nolimits_{b \in \mathcal{A}} 
		{{{\left\| {{v_b}} 
					\right\|}^2}} } \right] + 
	2{\mathbb{E}_\mathcal{A}}\left[ 
	{\sum\nolimits_{1 
			\le b < A} 
		{\left\langle {{v_b},\sum\nolimits_{b < k \le A} {{v_k}} } 
			\right\rangle } } \right]\nonumber\\
	=& B\frac{1}{n}\sum\nolimits_{i = 1}^n {{{\left\| {{v_i}} 
				\right\|}^2}}  + 2{\mathbb{E}_\mathcal{A}}\left[ 
	{\sum\nolimits_{1 \le b 
			< A} 
		{\left\langle {\mathbb{E}\left[ v \right],\sum\nolimits_{b < k 
					\le 
					A} 
				{{v_k}} } 
			\right\rangle } } \right]\nonumber\\
	=& A\frac{1}{n}\sum\nolimits_{i = 1}^n {{{\left\| {{v_i}} 
				\right\|}^2}}  + A\left( {A - 1} \right){\left\| 
		{\mathbb{E}\left[ v 
			\right]} 
		\right\|^2}\\
	=& A\frac{1}{n}\sum\nolimits_{i = 1}^n {{{\left\| {{v_i}} 
				\right\|}^2}} 
	\nonumber.
	\end{align} 	
	
\end{proof}

\begin{lemma}\label{VRNonCS-SCSG:Lemma:GeometriProgression}
	For the sequences that satisfy ${c_{k - 1}} = {c_k}Y + U$, where $Y>1$, $U>0$, $k\ge 1$ and 
	$c_0>0$, we can get the geometric progression
	
	\centering{${c_k} + \frac{U}{{Y - 1}} = \frac{1}{Y}\left( {{c_{k - 1}} + \frac{U}{{Y - 1}}} 
		\right),$}
	\leftline{then $c_k$ can be represented as decrease sequences,}
	\centering{${c_k} = {\left( {\frac{1}{Y}} \right)^k}\left( {{c_0} + \frac{U}{{Y - 1}}} \right) 
		- \frac{U}{{Y - 1}}.$}
\end{lemma}
\section{Bound analysis of SC-SCSG for the composition problem}
\begin{lemma2}\ref{VRNonCS-SCSG:Lemma:Bound-estimate-G-2}
	Suppose Assumption 
	\ref{VRNonCS-SCSG:Assumption:G} and \ref{VRNonCS-SCSG:Assumption:Middle-Bound} holds, for 
	${{\hat G}_k}$ defined in (\ref{VRNonCS-SCSG:Definition:Estimate-G}) with $D =\left|{\cal 
		D}_1\right|$ 
	and  $A =\left|{\cal A}\right|$, 
	we have
	\begin{align}
	{\mathbb{E}_{\mathcal{A},\mathcal{D}_1}}{\| {{{\hat G}_k} - G({x_k})} \|^2} \le 
	4\left( {\frac{{\mathbb{I}\left( {A < n} \right)}}{A} + \frac{{\mathbb{I}\left( {D < n} 
				\right)}}{D}} 
	\right)B_G^2\mathbb{E}{\left\| {{x_k} - {{\tilde x}_s}} \right\|^2} + 
	2\frac{{\mathbb{I}\left( {D < 
				n} 
			\right)}}{D}{H_1}.
	\end{align}
\end{lemma2}
\begin{proof}
	By the definition of ${{\hat G}_k}$  in (\ref{VRNonCS-SCSG:Definition:Estimate-G}), we have
	\begin{align*}
	\mathbb{E}{\| {{{\hat G}_k} - G({x_k})} \|^2} =& \mathbb{E}{\| {{{\hat 
					G}_k} - {G_{\mathcal{D}_{1}}}({x_k}) + {G_{\mathcal{D}_{1}}}({x_k}) - G({x_k})} 
		\|^2}\\
	\mathop  \le \limits^{\scriptsize \textcircled{\tiny{1}}}& 2\mathbb{E}{\| {{{\hat 
					G}_k} - {G_{\mathcal{D}_{1}}}({x_k})} \|^2} + 
	2\mathbb{E}{\left\| {{G_{\mathcal{D}_{1}}}({x_k}) - G({x_k})} \right\|^2}\\
	\mathop  \le \limits^{\scriptsize \textcircled{\tiny{2}}}& 4\left( {\frac{{\mathbb{I}\left( 
				{A < n} 
				\right)}}{A} + \frac{{\mathbb{I}\left( {D < n} \right)}}{D}} 
	\right)B_G^2\mathbb{E}{\left\| {{x_k} - {{\tilde x}_s}} \right\|^2} + 
	2\frac{{\mathbb{I}\left( {D < 
				n} 
			\right)}}{D}{H_1},
	\end{align*}
	where 
	${\small\textcircled{\scriptsize{1}}}$ follows from 
	$||a_1+a_2||^2\le 2a_1^2+2a_2^2$;
	${\small\textcircled{\scriptsize{2}}}$ is based on  Assumption 
	\ref{VRNonCS-SCSG:Assumption:Middle-Bound} and the following inequality: 
	Through adding and subtracting the term $ 
	{G({x_k}) - G({{\tilde x}_s})}$, we have
	\begin{align*}
	&\mathbb{E}{\| {{{\hat G}_k} - {G_{ \mathcal{D}_{1}}}({x_k})} \|^2}\\
	=& \mathbb{E}{\left\| {{G_{{{\cal A}}}}({x_k}) - {G_{{{\cal A}}}}({{\tilde x}_s}) + 
			{G_{\mathcal{D}_{1}}}({{\tilde x}_s}) - {G_{ \mathcal{D}_{1}}}({x_k})} \right\|^2}\\
	=& \mathbb{E}{\left\| {{G_{{{\cal A}}}}({x_k}) - {G_{{{\cal A}}}}({{\tilde x}_s}) - 
			\left( 
			{G({x_k}) - G({{\tilde x}_s})} \right) + \left( {G({x_k}) - G({{\tilde x}_s})} \right) 
			+ 
			{G_{ \mathcal{D}_{1}}}({{\tilde x}_s}) - {G_{\mathcal{D}_{1}}}({x_k})} \right\|^2}\\
	\mathop  \le \limits^{\scriptsize \textcircled{\tiny{1}}}& 2\mathbb{E}{\left\| {{G_{{{\cal 
							A}}}}({x_k}) - {G_{{{\cal A}}}}({{\tilde x}_s}) - 
			\left( 
			{G({x_k}) - G({{\tilde x}_s})} \right))} \right\|^2} + 2\mathbb{E}{\left\| 
		{{G_{\mathcal{D}_{1}}}({{\tilde 
					x}_s}) - {G_{\mathcal{D}_{1}}}({x_k}) - \left( {G({{\tilde x}_s}) - G({x_k})} 
			\right)} 
		\right\|^2}\\
	\mathop  \le \limits^{\scriptsize \textcircled{\tiny{2}}}& 2\frac{{\mathbb{I}\left( {A < n} 
			\right)}}{A}{\mathbb{E}_i}{\left\| {{G_i}({{\tilde x}_s}) - 
			{G_i}({x_k})} \right\|^2} + 2\frac{{\mathbb{I}\left( {D < n} 
			\right)}}{D}{\mathbb{E}_i}{\left\| 
		{{G_i}({{\tilde x}_s}) - {G_i}({x_k})} \right\|^2}\\
	\mathop  \le \limits^{\scriptsize \textcircled{\tiny{3}}}& 2\frac{{\mathbb{I}\left( {A < n} 
			\right)}}{A}B_G^2\mathbb{E}{\left\| {{x_k} - {{\tilde x}_s}} 
		\right\|^2} 
	+ 2\frac{{\mathbb{I}\left( {D < n} \right)}}{D}B_G^2\mathbb{E}{\left\| {{x_k} - {{\tilde 
					x}_s}} 
		\right\|^2}\\
	=& 2\left( {\frac{{\mathbb{I}\left( {A < n} \right)}}{A} + \frac{{\mathbb{I}\left( {D < n} 
				\right)}}{D}} 
	\right)B_G^2\mathbb{E}{\left\| {{x_k} - {{\tilde x}_s}} \right\|^2},
	\end{align*}
	where 
	${\small\textcircled{\scriptsize{1}}}$ follows from 
	$||a+b||^2\le 2a^2+2b^2$;
	${\small\textcircled{\scriptsize{2}}}$ is based on Assumption 
	\ref{VRNonCS-SCSG:Assumption:Middle-Bound};
	${\small\textcircled{\scriptsize{3}}}$ follows from the bounded function of $G$ in Assumption 
	\ref{VRNonCS-SCSG:Assumption:G}.
\end{proof}
\begin{lemma2}\ref{VRNonCS-SCSG:Lemma:Bound-estimate-G-3}
	Suppose 
	Assumption 
	\ref{VRNonCS-SCSG:Assumption:G}, \ref{VRNonCS-SCSG:Assumption:F}, 
	\ref{VRNonCS-SCSG:AssumptionIndependent} and
	\ref{VRNonCS-SCSG:Assumption:Middle-Bound} holds, for ${{\hat G}_k}$ defined in 
	(\ref{VRNonCS-SCSG:Definition:Estimate-G}) and ${\nabla {{\tilde f}_k}}$ defined in 
	(\ref{VRNonCS-SCSG:Definition:Estimate-f}) with $\mathcal{D} = \left[ 
	{{\mathcal{D}_1},{\mathcal{D}_2}} \right]$ and $D =\left|{\cal D}_1\right|=\left|{\cal 
		D}_2\right|$, we have
	\begin{align*}
	&{\mathbb{E}_{\mathcal{A,D}}}{\| {{\mathbb{E}_{i_k,j_k}}\left[ {\nabla {{\tilde f}_k}} 
			\right] - \nabla 
			f({x_k})} 
		\|^2}\\
	\le&4B_G^4L_F^2\left( {4\frac{{\mathbb{I}\left( {A < n} \right)}}{A} + 
		4\frac{{\mathbb{I}\left( {D < n} 
				\right)}}{D}} \right)\mathbb{E}{\left\| {{x_k} - {{\tilde x}_s}} \right\|^2} + 
	16B_G^2L_F^2\frac{{\mathbb{I}(D 
			< n)}}{D}{H_1} + 4\frac{{\mathbb{I}(D^2 < n^2)}}{{D^2}}{H_2},
	\end{align*}
\end{lemma2}
\begin{proof}
	Through adding and subtracting the terms of ${(\partial G({x_k}))^\mathsf{T}}\nabla 
	F(G({x_k})),{(\partial {G_{{{\cal D}_1}}}({{\tilde x}_s}))^\mathsf{T}}\nabla {F_{{{\cal 
					D}_1}}}(G({{\tilde x}_s})),{(\partial G({{\tilde x}_s}))^\mathsf{T}}\nabla 
	F(G({{\tilde 
			x}_s}))$, we have
	\begin{align*}
	&{\mathbb{E}_{\mathcal{A,D}}}{\| {{\mathbb{E}_{i_k,j_k}}\left[ {\nabla {{\tilde f}_k}} 
			\right] - \nabla 
			f({x_k})} 
		\|^2}\\
	=&\mathbb{E}{\left\| {{(\partial G({x_k}))^\mathsf{T}}\nabla F({{\hat G}_k}) - {(\partial 
				G({{\tilde 
						x}_s}))^\mathsf{T}}\nabla F({G_{{{\cal D}_1}}}({{\tilde x}_s})) + \nabla 
			{{\hat 
					f}_\mathcal{D}}({{\tilde 
					x}_s}) - \nabla f({x_k})} \right\|^2}\\
	\mathop  \le \limits^{\scriptsize \textcircled{\tiny{1}}}& 4\mathbb{E}{\left\| {{(\partial 
				G({x_k}))^\mathsf{T}}\nabla F({{\hat G}_k}) - {(\partial 
				G({x_k}))^\mathsf{T}}\nabla F(G({x_k}))} \right\|^2}\\
	&+ 4\mathbb{E}{\left\| {{(\partial G({{\tilde x}_s}))^\mathsf{T}}\nabla F(G({{\tilde 
					x}_s})) - 
			{(\partial 
				G({{\tilde x}_s}))^\mathsf{T}}\nabla F({G_{{{\cal D}_1}}}({{\tilde x}_s}))} 
		\right\|^2}\\
	&+ 4\mathbb{E}{\left\| {\nabla {{\hat f}_\mathcal{D}}({{\tilde x}_s}) - {(\partial 
				{G_{{{\cal 
								D}_1}}}({{\tilde 
						x}_s}))^\mathsf{T}}\nabla {F_{{{\cal D}_2}}}(G({{\tilde x}_s}))} 
		\right\|^2}\\
	&+ 4\mathbb{E}{\left\| {{(\partial {G_{{{\cal D}_1}}}({{\tilde x}_s}))^\mathsf{T}}\nabla 
			{F_{{{\cal 
							D}_2}}}(G({{\tilde x}_s})) - {(\partial G({{\tilde 
						x}_s}))^\mathsf{T}}\nabla F(G({{\tilde 
					x}_s}))} 
		\right\|^2}\\
	\mathop  \le \limits^{\scriptsize \textcircled{\tiny{2}}}& 4B_G^2L_F^2\mathbb{E}{\left\| 
		{{{\hat 
					G}_k} - G({x_k})} \right\|^2} + 4B_G^2L_F^2\mathbb{E}{\left\| 
		{G({{\tilde x}_s}) - {G_{{{\cal D}_1}}}({{\tilde x}_s})} \right\|^2} + 
	4B_G^2L_F^2\mathbb{E}{\left\| 
		{G({{\tilde x}_s}) - {G_{{{\cal D}_1}}}({{\tilde x}_s})} \right\|^2} + 
	4\frac{{\mathbb{I}(D^2 < n^2)}}{{D^2}}{H_2}\\
	\mathop  \le \limits^{\scriptsize \textcircled{\tiny{3}}}& 4B_G^4L_F^2\left( 
	{4\frac{{\mathbb{I}\left( {A < n} \right)}}{A} + 4\frac{{\mathbb{I}\left( {D < n} 
				\right)}}{D}} \right)\mathbb{E}{\left\| {{x_k} - {{\tilde x}_s}} \right\|^2} + 
	16B_G^2L_F^2\frac{{\mathbb{I}(D 
			< n)}}{D}{H_1} + 4\frac{{\mathbb{I}(D^2 < n^2)}}{{D^2}}{H_2},
	\end{align*}
	where ${\small\textcircled{\scriptsize{1}}}$ follows from $||a_1+a_2+a_3+a_4||^2\le 
	4a_1^2+4a_2^2+4a_3^2+4a_4^2$;
	${\small\textcircled{\scriptsize{2}}}$ is based on the bounded Jacobian of $G$ and the 
	smoothness of $F$ in Assumption \ref{VRNonCS-SCSG:Assumption:G} and 
	\ref{VRNonCS-SCSG:Assumption:F}, and the upper 
	bound of variance  in Assumption 
	\ref{VRNonCS-SCSG:Assumption:Middle-Bound} and Lemma 
	\ref{VRNonCS-SCSG:Appendix:Lemma:Inequation-indepent-double}.
	${\small\textcircled{\scriptsize{3}}}$ is based on Lemma 
	\ref{VRNonCS-SCSG:Lemma:Bound-estimate-G-2} and Assumption 
	\ref{VRNonCS-SCSG:Assumption:Middle-Bound}.
\end{proof}
\begin{lemma2}\ref{VRNonCS-SCSG:Lemma:Bound-estimate-G-4}
	Suppose 
	Assumption \ref{VRNonCS-SCSG:Assumption:G}-\ref{VRNonCS-SCSG:Assumption:Middle-Bound} holds, 
	for 
	${{\hat 
			G}_k}$ defined in (\ref{VRNonCS-SCSG:Definition:Estimate-G}) and ${\nabla {{\tilde 
				f}_k}}$ defined in (\ref{VRNonCS-SCSG:Definition:Estimate-f}) with $\mathcal{D} = 
	\left[ 
	{{\mathcal{D}_1},{\mathcal{D}_2}} \right]$ and $D =\left|{\cal 
		D}_1\right|=\left|{\cal 
		D}_2\right|$, we have
	\begin{align*}
	&\mathbb{E}_{i_k,j_k,\cal A, \cal D}{\| {\nabla {{\tilde f}_k} - \nabla f\left( {{x_k}} 
			\right)} \|^2}\\
	\le&5B_G^4L_F^2\left( 
	{\frac{L_f^2}{B_G^4L_F^2} + 
		4\frac{{\mathbb{I}\left( {A < n} \right)}}{A} + 4\frac{{\mathbb{I}\left( {D < n} 
				\right)}}{D}} \right)\mathbb{E}{\left\| {{x_k} - {{\tilde x}_s}} \right\|^2} + 
	20B_G^2L_F^2\frac{{\mathbb{I}(D < 
			n)}}{D}{H_1} + 5\frac{{\mathbb{I}(D^2 < n^2)}}{{D^2}}{H_2},
	\end{align*}
\end{lemma2}

\begin{proof}
	Through adding and subtracting the term of ${(\partial {G_j}({x_k}))^\mathsf{T}}\nabla 
	{F_i}(G({x_k}))$, ${(\partial {G_j}({{\tilde x}_s}))^\mathsf{T}}\nabla {F_i}(G({{\tilde 
			x}_s}))$, 
	${(\partial 
		G({{\tilde x}_s}))^\mathsf{T}}\nabla F(G({{\tilde x}_s}))$, ${(\partial {G_{{{\cal 
						D}_1}}}({{\tilde 
				x}_s}))^\mathsf{T}}\nabla {F_{{{\cal D}_1}}}(G({{\tilde x}_s}))$, we have
	\begin{align*}
	&\mathbb{E}{\| {\nabla {{\tilde f}_k} - \nabla f\left( {{x_k}} \right)} \|^2}\\
	=&\mathbb{E}{\left\| {{(\partial G_j({x_k}))^\mathsf{T}}\nabla F_i({{\hat G}_k}) - {(\partial 
				G_j({{\tilde 
						x}_s}))^\mathsf{T}}\nabla F_i({G_{\mathcal{D}_{1}}}({{\tilde x}_s})) + 
			\nabla 
			{\hat f_{\cal 
					D}}({{\tilde x}_s}) 
			- 
			\nabla f({x_k})} \right\|^2}\\
	\mathop  \le \limits^{\scriptsize \textcircled{\tiny{1}}}& 5\mathbb{E}{\left\| {{(\partial 
				{G_j}({x_k}))^\mathsf{T}}\nabla {F_i}(G({x_k})) - {(\partial 
				{G_j}({{\tilde x}_s}))^\mathsf{T}}\nabla {F_i}(G({{\tilde x}_s})) - \left( {\nabla 
				f({x_k}) - 
				{(\partial G({{\tilde x}_s}))^\mathsf{T}}\nabla F(G({{\tilde x}_s}))} \right)} 
		\right\|^2}\\
	&+ 5\mathbb{E}{\left\| {{(\partial {G_j}({x_k}))^\mathsf{T}}\nabla {F_i}({{\hat G}_k}) - 
			{(\partial 
				{G_j}({x_k}))^\mathsf{T}}\nabla {F_i}(G({x_k}))} \right\|^2}\\
	&+ 5\mathbb{E}{\left\| {{(\partial {G_j}({{\tilde x}_s}))^\mathsf{T}}\nabla {F_i}(G({{\tilde 
					x}_s})) - 
			{(\partial {G_j}({{\tilde x}_s}))^\mathsf{T}}\nabla {F_i}({G_{{{\cal D}_1}}}({{\tilde 
					x}_s}))} 
		\right\|^2}\\
	&+ 5\mathbb{E}{\left\| {\nabla {{\hat f}_\mathcal{D}}({{\tilde x}_s}) - {(\partial {G_{{{\cal 
								D}_1}}}({{\tilde 
						x}_s}))^\mathsf{T}}\nabla {F_{{{\cal D}_2}}}(G({{\tilde x}_s}))} 
		\right\|^2}\\
	&+ 5\mathbb{E}{\left\| {{(\partial {G_{{{\cal D}_1}}}({{\tilde x}_s}))^\mathsf{T}}\nabla 
			{F_{{{\cal 
							D}_2}}}(G({{\tilde x}_s})) - {(\partial G({{\tilde 
						x}_s}))^\mathsf{T}}\nabla F(G({{\tilde 
					x}_s}))} 
		\right\|^2}\\
	\mathop  \le \limits^{\scriptsize \textcircled{\tiny{2}}}& 5L_f^2\mathbb{E}{\left\| {{x_k} - 
			{{\tilde x}_s}} \right\|^2} + 5B_G^2L_F^2\mathbb{E}{\left\| {{{\hat 
					G}_k} - G({x_k})} \right\|^2} + 5B_G^2L_F^2\mathbb{E}{\left\| {G({{\tilde 
					x}_s}) - {G_{{{\cal 
							D}_1}}}({{\tilde x}_s})} \right\|^2}\\& + 5B_G^2L_F^2\mathbb{E}{\left\| 
		{G({{\tilde x}_s}) - 
			{G_{{{\cal 
							D}_1}}}({{\tilde x}_s})} \right\|^2} + 5\frac{{\mathbb{I}(D^2 < 
			n^2)}}{{D^2}}{H_2}\\
	\mathop  \le \limits^{\scriptsize \textcircled{\tiny{3}}}& 5B_G^4L_F^2\left( 
	{\frac{L_f^2}{B_G^4L_F^2} + 
		4\frac{{\mathbb{I}\left( {A < n} \right)}}{A} + 4\frac{{\mathbb{I}\left( {D < n} 
				\right)}}{D}} \right)\mathbb{E}{\left\| {{x_k} - {{\tilde x}_s}} \right\|^2} + 
	20B_G^2L_F^2\frac{{\mathbb{I}(D < 
			n)}}{D}{H_1} + 5\frac{{\mathbb{I}(D^2 < n^2)}}{{D^2}}{H_2},
	\end{align*}	
	where ${\small\textcircled{\scriptsize{1}}}$ follows from $||a_1+a_2+a_3+a_4+a_5||^2\le 
	5a_1^2+5a_2^2+5a_3^2+5a_4^2+5a_5^2$;
	${\small\textcircled{\scriptsize{2}}}$ is based on $
	\mathbb{E}[ \| X - \mathbb{E}[ X ] \|^2 ]  =  \mathbb{E}[ X^2 - \| \mathbb{E}[ 
	X ] \|^2 ] \le \mathbb{E}[ X^2 ]$, the smoothness of $F_i$ in Assumption 
	\ref{VRNonCS-SCSG:Assumption:GF}, the bounded Jacobian of $G(x)$ and the 
	smoothness of $F$ in Assumption \ref{VRNonCS-SCSG:Assumption:G} and 
	\ref{VRNonCS-SCSG:Assumption:F}, and the upper 
	bound of variance  in Assumption 
	\ref{VRNonCS-SCSG:Assumption:Middle-Bound} and Lemma 
	\ref{VRNonCS-SCSG:Appendix:Lemma:Inequation-indepent-double}.
	${\small\textcircled{\scriptsize{3}}}$ is based on Lemma 
	\ref{VRNonCS-SCSG:Lemma:Bound-estimate-G-2} and Assumption 
	\ref{VRNonCS-SCSG:Assumption:Middle-Bound}.
\end{proof}
\section{Proof of SC-SCSG method for Convex composition problem}
\begin{lemma}
	For $f(x)$ is $\mu$-strongly convex, by setting $\eta  = O\left( {\frac{\mu }{{L_f^2}}} 
	\right), k = O\left( {\frac{1}{\eta }} 
	\right) = O\left( {\frac{{L_f^2}}{\mu }} \right)$, we have the geometric convergence in 
	expectation: 
	\begin{align*}
	E{\left\| {{{\tilde x}_{s + 1}} - {x^*}} \right\|^2} \le {\rho ^s}E{\left\| {{{\tilde x}_0} 
			- {x^*}} \right\|^2}
	\end{align*}
	where $\rho  = \frac{1}{{2\left( {\mu  - 2L_f^2\eta } \right)\eta K}} + \frac{{L_f^2\eta 
	}}{{\left( {\mu  - 2L_f^2\eta } \right)}} < 1$. The gradient complexity is 
	\begin{align}
	O\left( {\left( {n + K} \right)\log \left( {1/\varepsilon } \right)} \right) = O\left( 
	{\left( {n + \frac{{L_f^2}}{\mu }} \right)\log \left( {1/\varepsilon } \right)} \right)
	\end{align}
\end{lemma}
\begin{proof}
	\begin{align*}
	&{E_{i.j}}{\left\| {{x_{k + 1}} - {x^*}} \right\|^2}\\
	=& E{\left\| {{x_k} - {x^*}} \right\|^2} - 2\eta E\langle \nabla {{\tilde f}_k},{x_k} - 
	{x^*}\rangle  + {\eta ^2}E{\left\| {\nabla {{\hat f}_k}} \right\|^2}\\
	=& E{\left\| {{x_k} - {x^*}} \right\|^2} - 2\eta E\langle f\left( {{x_k}} \right),{x_k} - 
	{x^*}\rangle  + {\eta ^2}E{\left\| {\nabla {{\hat f}_k}} \right\|^2}\\
	\le&  E{\left\| {{x_k} - {x^*}} \right\|^2} - 2\eta \mu E{\left\| {{x_k} - {x^*}} 
		\right\|^2} + 2{\eta ^2}E{\left\| {\nabla f({x_k})} \right\|^2} + 2{\eta ^2}E{\left\| 
		{\nabla {{\tilde f}_k} - \nabla f({x_k})} \right\|^2}\\
	=& E{\left\| {{x_k} - {x^*}} \right\|^2} - 2\eta \mu E{\left\| {{x_k} - {x^*}} \right\|^2} 
	+ {\eta ^2}\left( {2E{{\left\| {\nabla f({x_k}) - \nabla f({x^*})} \right\|}^2} + 
		2L_f^2E{{\left\| {{x_k} - {{\tilde x}_s}} \right\|}^2}} \right)\\
	\le& E{\left\| {{x_k} - {x^*}} \right\|^2} - 2\eta \mu E{\left\| {{x_k} - {x^*}} 
		\right\|^2} + {\eta ^2}\left( {2L_f^2E{{\left\| {{x_k} - {x^*}} \right\|}^2} + 
		2L_f^2E{{\left\| {{x_k} - {x^*}} \right\|}^2} + 2L_f^2E{{\left\| {{{\tilde x}_s} - {x^*}} 
				\right\|}^2}} \right)\\
	=& E{\left\| {{x_k} - {x^*}} \right\|^2} - 2\left( {\mu  - 2L_f^2\eta } \right)\eta 
	E{\left\| {{x_k} - {x^*}} \right\|^2} + 2L_f^2{\eta ^2}E{\left\| {{{\tilde x}_s} - {x^*}} 
		\right\|^2}
	\end{align*}
	Summing up from $k=0$ to $k=K-1$, we have
	\begin{align*}
	E{\left\| {{x_K} - {x^*}} \right\|^2} \le E{\left\| {{x_0} - {x^*}} \right\|^2} - 2\left( 
	{\mu  - 2L_{}^2\eta } \right)\eta KE{\left\| {{{\tilde x}_{s + 1}} - {x^*}} \right\|^2} + 
	2L_f^2{\eta ^2}KE{\left\| {{{\tilde x}_s} - {x^*}} \right\|^2}
	\end{align*}
	For $x_0=\tilde{x}_s$, we have
	\begin{align*}
	E{\left\| {{{\tilde x}_{s + 1}} - {x^*}} \right\|^2} \le& \frac{{1 + 2L_f^2{\eta 
				^2}K}}{{2\left( {\mu  - 2L_f^2\eta } \right)\eta K}}E{\left\| {{{\tilde x}_s} - 
				{x^*}} 
		\right\|^2} - \frac{1}{{2\left( {\mu  - 2L_f^2\eta } \right)\eta K}}E{\left\| {{x_K} - 
		{x^*}} 
		\right\|^2}\\ 
	\le& \frac{{1 + 2L_f^2{\eta ^2}K}}{{2\left( {\mu  - 2L_f^2\eta } 
			\right)\eta K}}E{\left\| {{{\tilde x}_s} - {x^*}} \right\|^2}\\ 
	=& \left( 
	{\frac{1}{{2\left( 
				{\mu  - 2L_f^2\eta } \right)\eta K}} + \frac{{L_f^2\eta }}{{\left( {\mu  - 
					2L_f^2\eta } 
				\right)}}} \right)E{\left\| {{{\tilde x}_s} - {x^*}} \right\|^2}
	\end{align*}
	By setting $\eta  = O\left( {\frac{\mu }{{L_f^2}}} \right),K = O\left( {\frac{1}{\eta }} 
	\right) = O\left( {\frac{{L_f^2}}{\mu }} \right)$, we have the geometric convergence in 
	expectation:
	\begin{align}
	E{\left\| {{{\tilde x}_{s + 1}} - {x^*}} \right\|^2} \le {\rho ^s}E{\left\| {{{\tilde x}_0} 
			- {x^*}} \right\|^2}
	\end{align}
	where $\rho  = \frac{1}{{2\left( {\mu  - 2L_f^2\eta } \right)\eta K}} + \frac{{L_f^2\eta 
	}}{{\left( {\mu  - 2L_f^2\eta } \right)}} < 1$
\end{proof}

\textbf{Proof of Theorem \ref{VRNonCS-SCSG:Lemma:NewSequence}}
\begin{proof}
	By the update of $x_k$ in Algorithm \ref{VRNonCS-SCSG:AlgorithmI}, we have
	\begin{align*}
	&{\mathbb{E}_{i.j}}{\left\| {{x_{k + 1}} - {x^*}} \right\|^2}\\
	=& \mathbb{E}{\left\| {{x_k} - {x^*}} \right\|^2} - 2\eta \mathbb{E}\langle \nabla {{\tilde 
			f}_k},{x_k} - 
	{x^*}\rangle  + {\eta ^2}\mathbb{E}{\left\| {\nabla {{\hat f}_k}} \right\|^2}\\
	=& \mathbb{E}{\left\| {{x_k} - {x^*}} \right\|^2} - 2\eta \mathbb{E}\langle \nabla f({x_k}) + 
	{\mathbb{E}_{i,j}}\left[ 
	{\nabla {{\tilde f}_k}} \right] - \nabla f({x_k}),{x_k} - {x^*}\rangle  + {\eta 
		^2}\mathbb{E}{\left\| {\nabla {{\hat f}_k}} \right\|^2}\\
	=& \mathbb{E}{\left\| {{x_k} - {x^*}} \right\|^2} - 2\eta \mathbb{E}\langle \nabla 
	f({x_k}),{x_k} - {x^*}\rangle  
	- 2\eta \mathbb{E}\langle {\mathbb{E}_{i,j}}\left[ {\nabla {{\tilde f}_k}} \right] - \nabla 
	f({x_k}),{x_k} - 
	{x^*}\rangle  + {\eta ^2}\mathbb{E}{\left\| {\nabla {{\hat f}_k} + \nabla f({x_k}) - \nabla 
			f({x_k})} 
		\right\|^2}\\
	\mathop  \le \limits^{\scriptsize \textcircled{\tiny{1}}}& \mathbb{E}{\left\| {{x_k} - {x^*}} 
		\right\|^2} - 2\eta \mu \mathbb{E}{\left\| {{x_k} - 
			{x^*}} \right\|^2} + 
	\eta \frac{1}{h}{\mathbb{E}_{\mathcal{A,D}}}{\left\| {{\mathbb{E}_{i,j}}\left[ {\nabla {{\tilde 
						f}_k}} 
			\right] - \nabla 
			f({x_k})} \right\|^2} + h\eta \mathbb{E}{\left\| {{x_k} - {x^*}} \right\|^2}\\
	&+ 2{\eta ^2}\left( {\mathbb{E}{{\left\| {\nabla f({x_k})} \right\|}^2} + \mathbb{E}{{\left\| 
				{\nabla {{\tilde 
							f}_k} - \nabla f({x_k})} \right\|}^2}} \right)\\
	=& \mathbb{E}{\left\| {{x_k} - {x^*}} \right\|^2} - \left( {2\eta \mu  - h\eta } 
	\right)\mathbb{E}{\left\| {{x_k} 
			- {x^*}} \right\|^2} + \eta \frac{1}{h}{\mathbb{E}_{\mathcal{A,D}}}{\left\| 
		{{\mathbb{E}_{i,j}}\left[ {\nabla {{\tilde 
						f}_k}} \right] - \nabla f({x_k})} \right\|^2}\\
	&+ 2{\eta ^2}\left( {\mathbb{E}{{\left\| {\nabla f({x_k}) - \nabla f({x^*})} \right\|}^2} + 
		\mathbb{E}{{\left\| 
				{\nabla {{\tilde f}_k} - \nabla f({x_k})} \right\|}^2}} \right)\\
	\mathop  \le \limits^{\scriptsize \textcircled{\tiny{2}}}& \mathbb{E}{\left\| {{x_k} - {x^*}} 
		\right\|^2} - \left( {2\eta \mu  - h\eta } 
	\right)\mathbb{E}{\left\| 
		{{x_k} - {x^*}} \right\|^2} + \eta \frac{1}{h}\left( {4V{{\left\| {{x_k} - {{\tilde x}_s}} 
				\right\|}^2} + {V_2}} \right)\\
	&+ 2{\eta ^2}\left( {L_f^2\mathbb{E}{{\left\| {{x_k} - {x^*}} \right\|}^2} + 5\left( {L_f^2 + 
			V} 
		\right){{\left\| {{x_k} - {{\tilde x}_s}} \right\|}^2} + {V_1}} \right)\\
	\mathop  \le \limits^{\scriptsize \textcircled{\tiny{3}}}& \mathbb{E}{\left\| {{x_k} - {x^*}} 
		\right\|^2} - \left( {2\mu  - h - 4V\frac{1}{h} - 
		\left( 
		{12L_f^2 + 10V} \right)\eta } \right)\eta \mathbb{E}{\left\| {{x_k} - {x^*}} \right\|^2}\\
	&+ 2\left( {2V\frac{1}{h} + 5\left( {L_f^2 + V} \right)\eta } \right)\eta \mathbb{E}{\left\| 
		{{{\tilde x}_s} - {x^*}} \right\|^2} + \frac{1}{h}\eta {V_2} + 2{\eta ^2}{V_1},
	\end{align*}	
	where
	\begin{align*}
	V =& B_G^4L_F^2\left( {4\frac{{\mathbb{I}\left( {A < n} \right)}}{A} + 4\frac{{\mathbb{I}\left( 
				{D < n} 
				\right)}}{D}} \right),\\
	{V_1} =& 20B_G^2L_F^2\frac{{\mathbb{I}(D < n)}}{D}{H_1} + 5\frac{{\mathbb{I}({D^2} < 
			{n^2})}}{{{D^2}}}{H_2},\\
	{V_2} =& 16B_G^2L_F^2\frac{{\mathbb{I}(D < n)}}{D}{H_1} + 4\frac{{\mathbb{I}({D^2} < 
			{n^2})}}{{{D^2}}}{H_2} = 
	\frac{4}{5}{V_1},
	\end{align*}
	${\small\textcircled{\scriptsize{1}}}$ is based on $||a_1+a_2||^2\le 2a_1^2+2a_2^2$ and 
	$\left\langle {a_1,a_2} \right\rangle \le h||a_1||^2+\frac{1}{h}||a_2||^2, h>0$;
	${\small\textcircled{\scriptsize{2}}}$ is based on strongly-convex of $f$ in Assumption 
	\ref{VRNonCS-SCSG:Assumption:f-strong};
	${\small\textcircled{\scriptsize{3}}}$ following from Lemma 
	\ref{VRNonCS-SCSG:Lemma:Bound-estimate-G-3} and Lemma 
	\ref{VRNonCS-SCSG:Lemma:Bound-estimate-G-4}.
	
	Summing up from $k=0$ to $k=K-1$, we have
	\begin{align*}
	\mathbb{E}{\left\| {{x_K} - {x^*}} \right\|^2} \le \mathbb{E}{\left\| {{x_0} - {x^*}} 
		\right\|^2} - 
	{\rho 
		_1}K\mathbb{E}{\left\| {{{\tilde x}_{s + 1}} - {x^*}} \right\|^2} + {\rho 
		_2}K\mathbb{E}{\left\| 
		{{{\tilde 
					x}_s} - {x^*}} \right\|^2} + {\rho _3}K,
	\end{align*}
	where 
	\begin{align*}
	{\rho _1} =& \left( {2\mu  - h - 4V\frac{1}{h} - \left( {12L_f^2 + 10V} \right)\eta } 
	\right)\eta=\left( {\mu  - 2L_f^2\eta } \right)\eta  - {\rho _2}, \\
	{\rho _2} =& 2\left( {2V\frac{1}{h} + 5\left( {L_f^2 + V} \right)\eta } \right)\eta, \\
	{\rho _3} =& \frac{1}{h}\eta {V_2} + 2{\eta ^2}{V_1}.
	\end{align*}
	For $x_0=\tilde{x}_s$, by arrange, we have
	\begin{align*}
	{\rho _1}\mathbb{E}{\left\| {{{\tilde x}_{s + 1}} - {x^*}} \right\|^2} \le& 
	\frac{1}{K}\mathbb{E}{\left\| 
		{{x_0} - {x^*}} \right\|^2} + {\rho _2}E{\left\| {{{\tilde x}_s} - {x^*}} \right\|^2} + 
	{\rho _3} - \frac{1}{K}\mathbb{E}{\left\| {{x_K} - {x^*}} \right\|^2}\\
	\le& \left( {\frac{1}{K} + {\rho _2}} \right)\mathbb{E}{\left\| {{{\tilde x}_s} - {x^*}} 
		\right\|^2} 
	+ {\rho _3}.
	\end{align*}
\end{proof}
\textbf{Proof of Corollary \ref{VRNonCS-SCSG:Lemma:paraeters setting}}
\begin{proof}
	In order to keep the proposed algorithm converge, we consider the parameters' setting, we first 
	ensure that $\rho_1>0$ in (\ref{VRNonCS-SCSG:Lemma:NewSequence-r1}), and then define
	\begin{align}\label{VRNonCS-SCSG:Lemma:paraeters setting-rho}
	\rho=(\frac{1}{K}+\rho_2)/\rho_1,
	\end{align} 
	that require $\rho<1$, where $\rho_2$ defined in (\ref{VRNonCS-SCSG:Lemma:NewSequence-r2}). 
	Thus, the convergence sequence  is 
	\begin{align*}
	\mathbb{E}{\left\| {{{\tilde x}_S} - {x^*}} \right\|^2} \le {\rho ^S}\mathbb{E}{\left\| 
		{{{\tilde x}_0} - 
			{x^*}} \right\|^2} + \frac{{{\rho _3}}}{{{\rho _1}}}\sum\limits_{s = 0}^S {{\rho ^s}}  
	\le 
	{\rho ^S}\mathbb{E}{\left\| {{{\tilde x}_0} - {x^*}} \right\|^2} + \frac{{{\rho 
				_3}}}{{{\rho 
				_1}}}\frac{1}{{1 - \rho }}.
	\end{align*}
	We ensure  $\frac{{{\rho _3}}}{{{\rho _1}}}\frac{1}{{1 - \rho }} \le \frac{1}{2}\varepsilon $, 
	where $\rho_3$ defined in (\ref{VRNonCS-SCSG:Lemma:NewSequence-r3}), 
	that we can derive the size of the $D$. In the following we analyze the parameters' setting 
	such that satisfying above requirement.
	\begin{enumerate}
		\item In order to ensure ${\rho _1} > 0$ in (\ref{VRNonCS-SCSG:Lemma:NewSequence-r1}), we 
		consider the parameter $h$, $\eta$ and $A$,
		\begin{enumerate}
			\item $h = \mu$, consider $\rho_1$ in (\ref{VRNonCS-SCSG:Lemma:NewSequence-r1}), we 
			should require that 
			$h\le \mu$, however, V in (\ref{VRNonCS-SCSG:Lemma:NewSequence-V}) has the relationship 
			with $A$ and $D$. In order to keep $A$ small enough, we set the upper bound of $h$. 
			Thus, we set $h=\mu$.
			\item $ A = \min \left\{ 
			{n,128B_G^4L_F^2\frac{1}{{{\mu ^2}}}} \right\}$, based on the setting of h, we require 
			that $V/h<\frac{\mu}{16}$. Thus, we have
			\begin{center}
				$V = B_G^4L_F^2\left( {4\frac{{I\left( {A < n} \right)}}{A} + 4\frac{{I\left( {D 
								< n} \right)}}{D}} \right) \le 8B_G^4L_F^2\frac{{I\left( {A < n} 
						\right)}}{A} \le 
				\frac{1}{16}{\mu ^2}$.
			\end{center}			
			For V defined in (\ref{VRNonCS-SCSG:Lemma:NewSequence-V}), if $A<n$, we have
			\begin{center}
				$A \ge 128B_G^4L_F^2\frac{1}{{{\mu ^2}}}$,
			\end{center}
			otherwise, $A=n$ satisfy the requirement. Thus, we have $ A = \min \left\{ 
			{n,128B_G^4L_F^2\frac{1}{{{\mu ^2}}}} \right\}$.
			\item $\eta  \le \frac{3\mu }{{53L_f^2}}$, back to the target of  $\rho_1>0$, we 
			require that $\eta  \le \frac{3\mu }{{53L_f^2}} 
			\le 
			\frac{{\frac{3}{4}\mu }}{{12L_f^2 + \frac{{10}}{8}L_f^2}} \le \frac{{\frac{3}{4}\mu 
			}}{{12L_f^2 + \frac{{10}}{8}{\mu ^2}}} = \frac{{\mu  - 4\frac{1}{\mu }V}}{{12L_f^2 + 
					10V}}= \frac{{2\mu  - h - 4\frac{1}{h}V}}{{2L_f^2 + 
					10\left( {L_f^2 + V} \right)}} $, note that $\mu\le L_f$ by the definition in 
			preliminaries. 
		\end{enumerate}
		\item In order to ensure $\rho<1$ in (\ref{VRNonCS-SCSG:Lemma:paraeters setting-rho}), we 
		first
		consider $\rho_1$ and $\rho_2$ in 
		(\ref{VRNonCS-SCSG:Lemma:NewSequence-r1}) and (\ref{VRNonCS-SCSG:Lemma:NewSequence-r2}).
		By the setting of $h=\mu$ and $V<\mu^2/16$, we have,
		\begin{align}
		\label{VRNonCS-SCSG:Lemma:paraeters setting-rho-1-bound}
		{\rho _1} \ge& \left( {\mu  - 2L_f^2\eta  - \left( {\frac{1}{4}\mu  + 10\left( {L_f^2 + 
					\frac{1}{{16}}{\mu ^2}} \right)\eta } \right)} \right)\eta  \ge \left( 
		{\frac{3}{4}\mu  - 
			\frac{{101}}{8}L_f^2\eta } \right)\eta, \\
		\label{VRNonCS-SCSG:Lemma:paraeters setting-rho-2-bound}
		{\rho _2} \le& 4\frac{1}{\mu }\eta \frac{1}{{16}}{\mu ^2} + 10\left( {L_f^2 + 
			\frac{1}{{16}}{\mu ^2}} \right){\eta ^2} \le \left( {\frac{1}{4}\mu  + 10\left( {L_f^2 
				+ 
				\frac{1}{{16}}{\mu ^2}} \right)\eta } \right)\eta  \ge \left( {\frac{1}{4}\mu  + 
			\frac{{85}}{8}L_f^2\eta } \right)\eta. 
		\end{align}
		We require that $\rho=\frac{1}{{K{\rho _1}}} + \frac{{{\rho _2}}}{{{\rho _1}}} < 1$, and 
		analyze 
		the two term separately,
		\begin{enumerate}
			\item In order to $\frac{{{\rho _2}}}{{{\rho _1}}} < \frac{1}{2}$, that is
			\begin{align*}
			\frac{{{\rho _2}}}{{{\rho _1}}} < \frac{{\left( {\frac{1}{4}\mu  + 
						\frac{{85}}{8}L_f^2\eta } \right)\eta }}{{\left( {\frac{3}{4}\mu  - 
						\frac{{101}}{8}L_f^2\eta } \right)\eta }} < \frac{1}{2}.
			\end{align*}
			We get $\eta  \le \frac{\mu }{{135L_f^2}}$.
			\item In order to $\frac{1}{{K{\rho _1}}} < \frac{1}{2}$, that is 
			\begin{align*}
			\frac{1}{{K{\rho _1}}} < \frac{1}{{2K{\rho _2}}} \le& \frac{1}{{2K\left( 
					{\frac{1}{4}\mu  + 10\left( {L_f^2 + \frac{1}{{16}}{\mu ^2}} \right)\eta } 
					\right)\eta }}\\
			\le& \frac{1}{{2K\left( {\frac{1}{4}\mu  + \frac{{85}}{8}L_f^2\eta } \right)\eta }} 
			\le \frac{1}{{2K\left( {\frac{1}{4}\mu \eta } \right)}} < \frac{1}{2}.
			\end{align*}
			Thus, we have $K \ge 540\frac{{L_f^2}}{{{\mu ^2}}}$.
		\end{enumerate}
		\item Consider the term ${\rho ^S}E{\left\| {{{\tilde x}_0} - {x^*}} \right\|^2} + 
		\frac{{{\rho _3}}}{{{\rho _1}}}\frac{1}{{1 - \rho }}$, we analyze them separately,
		\begin{enumerate}
			\item 	In order to ensure $\frac{{{\rho _3}}}{{{\rho _1}}}\frac{1}{{1 - \rho }} \le 
			\frac{1}{2}\varepsilon $, that is 
			\begin{align*}
			\frac{{{\rho _3}}}{{{\rho _1}}}\frac{1}{{1 - \left( {\frac{1}{{K{\rho _1}}} + 
						\frac{{{\rho _2}}}{{{\rho _1}}}} \right)}} = \frac{{{\rho _3}}}{{{\rho _1} 
					- 
					\frac{1}{K} - {\rho _2}}} \le \frac{{{\rho _3}}}{{{\rho _1} - \frac{1}{K} - 
					\frac{1}{2}{\rho _1}}} \le \frac{{{\rho _3}}}{{\frac{1}{2}{\rho _1} - 
					\frac{1}{K}}} \le 
			\frac{{2{\rho _3}}}{{{\rho _1}}} \le \frac{1}{2}\varepsilon .
			\end{align*}
			Based on the bound of $\rho_1$ in (\ref{VRNonCS-SCSG:Lemma:paraeters 
				setting-rho-1-bound}),  the definition of $V_1$ in 
			(\ref{VRNonCS-SCSG:Lemma:NewSequence-V1}) and the step size $\eta$ mentioned above, we 
			have		
			\begin{enumerate}
				\item For $V$
				\begin{align*}
				2\frac{{\frac{1}{\mu }\eta {V_2} + 2{\eta ^2}{V_1}}}{{{\rho _1}}} = 
				2\frac{{\frac{1}{\mu }{V_2} + 2\eta {V_1}}}{{\frac{3}{4}\mu  - 
						\frac{{101}}{8}L_f^2\eta 
				}} = \frac{{\frac{4}{5}\frac{1}{\mu }{V_1} + 2\eta {V_1}}}{{\frac{3}{4}\mu  - 
						\frac{{101}}{8}L_f^2\eta }} = \frac{{\left( {\frac{4}{{5\mu }} + 2\eta } 
						\right){V_1}}}{{\frac{3}{4}\mu  - \frac{{101}}{8}L_f^2\eta }} \le 
				\varepsilon, 
				\end{align*}
				thus, we have
				\begin{align*}
				{V_1} \le \frac{4}{5}\varepsilon {\mu ^2} \le \frac{{\left( {\frac{3}{4} - 
							\frac{{101}}{8}\frac{1}{{135}}} \right)}}{{\frac{4}{5} + 
						\frac{2}{{135}}}}\varepsilon {\mu ^2} \le \frac{{\left( {\frac{3}{4} - 
							\frac{{101}}{8}\frac{1}{{135}}} \right)\mu }}{{\frac{4}{{5\mu }} + 
						2\frac{\mu 
						}{{135\mu _{}^2}}}}\varepsilon  \le \frac{{\frac{3}{4}\mu  - 
						\frac{{101}}{8}L_f^2\frac{\mu }{{135L_f^2}}}}{{\frac{4}{{5\mu }} + 
						2\frac{\mu 
						}{{135L_f^2}}}}\varepsilon  \le \frac{{\frac{3}{4}\mu  - 
						\frac{{101}}{8}L_f^2\eta 
				}}{{\left( {\frac{4}{{5\mu }} + 2\eta } \right)}}\varepsilon 
				\end{align*}
				\item If $D<n$, we can obtain  $D \ge \frac{5}{{4\varepsilon{\mu ^2}}}\left( 
				{20B_G^4L_F^2{H_1} + 5{H_2}} \right)$, otherwise $D=0$, the above inequality is 
				correct. Thus, we obtain $D = \min \left\{ {n,\left( {16B_G^4L_F^2{H_1} + 4{H_2}} 
					\right)\frac{5}{{4\varepsilon {\mu ^2}}}} \right\}$.
			\end{enumerate}
			\item In order to ensure ${\rho ^S}E{\left\| {{{\tilde x}_0} - {x^*}} \right\|^2} \le 
			\frac{1}{2}\varepsilon$, we need the number of the outer iterations 
			\begin{align*}
			S \ge \frac{1}{{\log \left( {1/\rho } \right)}}\log \frac{{2E{{\left\| {{{\tilde x}_0} 
								- 
								{x^*}} \right\|}^2}}}{\varepsilon }.
			\end{align*}
		\end{enumerate}
	\end{enumerate}
	All in all, we consider the query complexity based on above parameters' setting. For each outer 
	iteration, there will be $\left( {D + KA} \right)$ queries. Thus, the query complexity is 
	\begin{align*}
	{\left( {D + KA} \right)S} =
	\mathcal{O}\left( 
	{\left( {\min \left\{ {n,\frac{1}{{{\varepsilon\mu ^2}}}} \right\} + \frac{{L_f^2}}{{{\mu 
						^2}}}\min 
			\left\{ {n,\frac{1}{{{\mu ^2}}}} \right\}} \right)\log \left( {1/\varepsilon } \right)} 
	\right).
	\end{align*}
\end{proof}
\section{Proof of SC-SCSG method for Non-convex composition problem}
\textbf{Proof of Lemma \ref{VRNonCS-SCSG:Non:Lemma:NewSequence}}

\begin{proof}
	Consider the upper bound of $f(x_{k+1})$ and ${\left\| {{x_{k + 1}} - {{\tilde x}_s}} 
		\right\|^2}$,respectively,
	\begin{itemize}
		\item 	Base on the smoothness of $f$ in Assumption \ref{VRNonCS-SCSG:Assumption:GF} and 
		take expectation with respective to $i_k, j_k$, we 
		have
		\begin{align*}		
		&{\mathbb{E}_{i,j}}\left[ {f({x_{k + 1}})} \right]\\
		\le& \mathbb{E}\left[ {f({x_k})} \right] - \eta \mathbb{E}\langle \nabla f({x_k}),\nabla 
		{{\tilde 
				f}_k}\rangle  + \frac{{{L_f}}}{2}{\eta ^2}\mathbb{E}{\left\| {\nabla {{\tilde 
						f}_k}} \right\|^2}\\
		=&  \mathbb{E}\left[ {f({x_k})} \right] - \eta \mathbb{E}\langle \nabla f({x_k}),\nabla 
		{{\tilde f}_k} - 
		\nabla f({x_k}) + \nabla f({x_k})\rangle  + \frac{{{L_f}}}{2}{\eta ^2}\mathbb{E}{\left\| 
			{\nabla 
				{{\tilde f}_k}} \right\|^2}\\
		=& \mathbb{E}\left[ {f({x_k})} \right] - \eta \mathbb{E}\langle \nabla f({x_k}),\nabla 
		f({x_k})\rangle  - 
		\eta \langle \nabla f({x_k}),\mathbb{E}\left[ {\nabla {{\tilde f}_k}} \right] - \nabla 
		f({x_k})\rangle  + \frac{{{L_f}}}{2}{\eta ^2}\mathbb{E}{\left\| {\nabla {{\tilde f}_k} - 
				\nabla 
				f({x_k}) + \nabla f({x_k})} \right\|^2}\\
		\le& \mathbb{E}\left[ {f({x_k})} \right] - \eta \mathbb{E}{\left\| {\nabla f({x_k})} 
			\right\|^2} + 
		\frac{1}{2}\eta \mathbb{E}{\left\| {\nabla f({x_k})} \right\|^2} + \frac{1}{2}\eta 
		\mathbb{E}{\left\| 
			{{\mathbb{E}_{i,j}}\left[ {\nabla {{\tilde f}_k}} \right] - \nabla f({x_k})} 
			\right\|^2}\\& + 
		\frac{{{L_f}}}{2}{\eta ^2}\left( {2\mathbb{E}{{\left\| {\nabla f({x_k})} \right\|}^2} + 
			2\mathbb{E}{{\left\| 
					{\nabla {{\tilde f}_k} - \nabla f({x_k})} \right\|}^2}} \right)\\
		=& \mathbb{E}\left[ {f({x_k})} \right] - \frac{1}{2}\eta \mathbb{E}{\left\| {\nabla 
				f({x_k})} \right\|^2} + 
		\frac{1}{2}\eta \mathbb{E}{\left\| {{\mathbb{E}_{i,j}}\left[ {\nabla {{\tilde f}_k}} 
				\right] - \nabla 
				f({x_k})} \right\|^2} + {L_f}{\eta ^2}\left( {\mathbb{E}{{\left\| {\nabla f({x_k})} 
					\right\|}^2} + 
			\mathbb{E}{{\left\| {\nabla {{\tilde f}_k} - \nabla f({x_k})} \right\|}^2}} \right)\\
		=& \mathbb{E}\left[ {f({x_k})} \right] - \left( {\frac{1}{2}\eta  - {L_f}{\eta ^2}} 
		\right)\mathbb{E}{\left\| 
			{\nabla f({x_k})} \right\|^2} + \frac{1}{2}\eta \mathbb{E}{\left\| 
			{{\mathbb{E}_{i,j}}\left[ {\nabla {{\tilde 
							f}_k}} \right] - \nabla f({x_k})} \right\|^2} + {L_f}{\eta 
			^2}\mathbb{E}{\left\| {\nabla {{\tilde 
						f}_k} - \nabla f({x_k})} \right\|^2},
		\end{align*}
		where the last inequality is based on $||a_1+a_2||^2\le 2a_1^2+2a_2^2$.
		\item 	Base on the update  of $x_k$ in Algorithm \ref{VRNonCS-SCSG:AlgorithmI} and 	
		take expectation with respective to $i_k, j_k$, we have, 
		\begin{align*}
		&{\mathbb{E}_{i.j}}{\left\| {{x_{k + 1}} - {{\tilde x}_s}} \right\|^2}\\
		=& \mathbb{E}{\left\| {{x_k} - {{\tilde x}_s}} \right\|^2} - 2\eta \mathbb{E}\langle \nabla 
		{{\tilde 
				f}_k},{x_k} - {{\tilde x}_s}\rangle  + {\eta ^2}\mathbb{E}{\left\| {\nabla {{\tilde 
						f}_k}} 
			\right\|^2}\\
		=& \mathbb{E}{\left\| {{x_k} - {{\tilde x}_s}} \right\|^2} - 2\eta \mathbb{E}\langle \nabla 
		{{\tilde 
				f}_k} - \nabla f({x_k}) + \nabla f({x_k}),{x_k} - {{\tilde x}_s}\rangle ] + {\eta 
			^2}\mathbb{E}{\left\| {\nabla {{\tilde f}_k}} \right\|^2}\\
		=& \mathbb{E}{\left\| {{x_k} - {{\tilde x}_s}} \right\|^2} - 2\eta \mathbb{E}\langle \nabla 
		f({x_k}),{x_k} - {{\tilde x}_s}\rangle ] - 2\eta \langle \mathbb{E}\left[ {\nabla {{\tilde 
					f}_k}} \right] - \nabla f({x_k}),{x_k} - {{\tilde x}_s}\rangle ] + {\eta 
			^2}\mathbb{E}{\left\| 
			{\nabla {{\tilde f}_k} - \nabla f({x_k}) + \nabla f({x_k})} \right\|^2}\\
		\le& \mathbb{E}{\left\| {{x_k} - {{\tilde x}_s}} \right\|^2} + h\eta {\left\| {\nabla 
				f({x_k})} 
			\right\|^2} + h\eta {\left\| {\mathbb{E}\left[ {\nabla {{\tilde f}_k}} \right] - \nabla 
				f({x_k})} \right\|^2} + \frac{2}{h}\eta \mathbb{E}{\left\| {{x_k} - {{\tilde x}_s}} 
			\right\|^2} 
		\\&+ {\eta ^2}\left( {2\mathbb{E}{{\left\| {\nabla f({x_k})} \right\|}^2} + 
			2\mathbb{E}{{\left\| 
					{\nabla 
						{{\tilde f}_k} - \nabla f({x_k})} \right\|}^2}} \right)\\
		=& \left( {1 + \frac{2}{h}\eta } \right)\mathbb{E}{\left\| {{x_k} - {{\tilde x}_s}} 
			\right\|^2} 
		+ \left( {h\eta  + 2{\eta ^2}} \right)\mathbb{E}{\left\| {\nabla f({x_k})} \right\|^2} + 
		h\eta 
		\mathbb{E}{\left\| {\mathbb{E}\left[ {\nabla {{\tilde f}_k}} \right] - \nabla f({x_k})} 
			\right\|^2} 
		+ 
		2{\eta ^2}\mathbb{E}{\left\| {\nabla {{\tilde f}_k} - \nabla f({x_k})} \right\|^2},
		\end{align*}
		
		where the inequality is based on  $2\langle {a_1,b_2} \rangle  \le $$1/h\| a_1 \|^2$$ +$$ 
		h\| a_2 \|^2$, $\forall h>0$, and $||a_1+a_2||^2\le 2a_1^2+2a_2^2$.
	\end{itemize}
	Combine above equalities and Lemma \ref{VRNonCS-SCSG:Lemma:Bound-estimate-G-3}, 
	\ref{VRNonCS-SCSG:Lemma:Bound-estimate-G-4}, 	we form a Lyapunov function,
	\begin{align*}
	&\mathbb{E}[f({x_{k + 1}})] + {c_{k + 1}}\mathbb{E}{\left\| {{x_{k + 1}} - {{\tilde x}_s}} 
		\right\|^2}\\
	=& \mathbb{E}[f({x_k})] - \left( {\frac{1}{2}\eta  - {L_f}{\eta ^2}} \right){\left\| {\nabla 
			f({x_k})} \right\|^2} + \frac{1}{2}\eta {\left\| {\mathbb{E}\left[ {\nabla {{\tilde 
						f}_k}} 
			\right] - 
			\nabla f({x_k})} \right\|^2} + {L_f}{\eta ^2}{\left\| {\nabla {{\tilde f}_k} - \nabla 
			f({x_k})} \right\|^2}\\
	&+ {c_{k + 1}}\left( {\left( {1 + \frac{2}{h}\eta } \right)\mathbb{E}{{\left\| {{x_k} - 
					{{\tilde 
							x}_s}} \right\|}^2} + \left( {h\eta  + 2{\eta ^2}} \right){{\left\| 
				{\nabla 
					f({x_k})} 
				\right\|}^2} + h\eta {{\left\| {\mathbb{E}\left[ {\nabla {{\tilde f}_k}} \right] - 
					\nabla f({x_k})} 
				\right\|}^2} + 2{\eta ^2}{{\left\| {\nabla {{\tilde f}_k} - \nabla f({x_k})} 
				\right\|}^2}} 
	\right)\\
	=& \mathbb{E}[f({x_k})] + {c_{k + 1}}\left( {1 + \frac{2}{h}\eta } \right)\mathbb{E}{\left\| 
		{{x_k} 
			- 
			{{\tilde x}_s}} \right\|^2} - \left( {\left( {\frac{1}{2} - {c_{k + 1}}h} \right)\eta  
		- 
		\left( {{L_f} + 2{c_{k + 1}}} \right){\eta ^2}} \right){\left\| {\nabla f({x_k})} 
		\right\|^2}\\
	&+ \left( {{L_f}{\eta ^2} + 2{\eta ^2}{c_{k + 1}}} \right){\left\| {\nabla {{\tilde f}_k} - 
			\nabla f({x_k})} \right\|^2} + \left( {\frac{1}{2}\eta  + h\eta {c_{k + 1}}} 
	\right){\left\| {\mathbb{E}\left[ {\nabla {{\tilde f}_k}} \right] - \nabla f({x_k})} 
		\right\|^2}\\
	\le& \mathbb{E}[f({x_k})] + {c_k}\mathbb{E}{\left\| {{x_k} - {{\tilde x}_s}} \right\|^2} - 
	{u_k}{\left\| 
		{\nabla f({x_k})} \right\|^2} +J_k,
	\end{align*}
	where
	\begin{align*}
	{u_k} =& \left( {\left( {\frac{1}{2} - h{c_{k + 1}}} \right)\eta  - \left( {{L_f} + 2{c_{k 
					+ 1}}} \right){\eta ^2}} \right);\\
	{W_1} =& 20B_G^2L_F^2\frac{{\mathbb{I}(D < n)}}{D}{H_1} + 5\frac{{\mathbb{I}(D^2 < 
			n^2)}}{{D^2}}{H_2};\\
	{W_2} =& \frac{4}{5}W_1;\\
	J_k=& \left( {\frac{1}{2} + h{c_{k + 1}}} \right){W_2}\eta  + 
	\left( {{L_f} + 2{c_{k + 1}}} \right){W_1}{\eta ^2};\\
	W =& B_G^4L_F^2\left( {4\frac{{\mathbb{I}\left( {A < n} \right)}}{A} + 
		4\frac{{\mathbb{I}\left( {D < 
					n} 
				\right)}}{D}} \right);\\
	{c_k} =& {c_{k + 1}}\left( {1 + \left( {\frac{2}{h} + 4hW} \right)\eta  + 10\left( 
		{L_f^2 + W} \right){\eta ^2}} \right) + 2W\eta  + 5(L_f^2 + W){L_f}{\eta ^2}.
	\end{align*}
\end{proof}
Based on the above inequality with respect to the sequence $\mathbb{E}[f({x_k})] + 
{c_k}\mathbb{E}{\left\| {{x_k} - {{\tilde x}_s}} \right\|^2}$ and Algorithm 
\ref{VRNonCS-SCSG:AlgorithmI}, we can obtain the convergence form in which the parameters are not 
clear defined.

\textbf{Proof of Theorem \ref{VRNonCS-SCSG:Non:Lemma:Convergence form}}
\begin{proof} Based on the update for $c_k$ in (\ref{VRNonCS-SCSG:Non:Lemma:NewSequence-c_k}), we 
	can see that $c_k>c_{k+1}$. As $c_k$ is a decreasing sequence, we have $u_0<u_k$ and $J_k<J_0$
	Then, we get 
	\begin{align*}
	{u_0}\mathbb{E}[ {\| {\nabla f( {{x_k}} )} \|^2} ] \le E[ {f( {{x_k}} )} ] + {c_k}\mathbb{E}[ 
	{\| {{x_k} - \tilde x_s} \|^2} ] - ( {\mathbb{E}[ {f( {{x_{k + 1}}} )} ] + {c_{k + 
				1}}\mathbb{E}[ {\| {{x_{k + 1}} - \tilde x_s} \|^2} ]} )+J_0.
	\end{align*}
	Sum from $k=0$ to $k=K-1$, we can get
	\begin{align*}
	\frac{1}{K}\sum\limits_{k = 0}^{K-1} {u\mathbb{E}[ {\| {\nabla f( {{x_k}} )} \|^2} ]}  &\le 
	\frac{{\mathbb{E}[ {f( {{x_0}} )} ] - ( {\mathbb{E}[ {f( {{x_K}} )} ] + {c_K}\mathbb{E}[ {\| 
					{{x_K} - \tilde x_s} \|^2} ]} )}}{{K}}+J_0\\
	&\le \frac{{\mathbb{E}[ {f( {{x_0}} )} ] -\mathbb{E}[ {f( {{x_K}} )} ]}}{{K}}+J_0.
	\end{align*}
	Since $x_0=\tilde{x}_s$, let $\tilde{x}_{s+1}=x_K$, we obtain,
	\begin{align*}
	\frac{1}{K}\sum\limits_{k = 0}^{K-1} {u_0\mathbb{E}[ {\| {\nabla f( {{x_k}} )} \|^2} ]} \le 
	\frac{{\mathbb{E}[ {f( {{{\tilde x}_s}} )} ] - \mathbb{E}[ {f( {{{\tilde x}_{s + 1}}} )} 
			]}}{K}+J_0.
	\end{align*} 
	Summing the outer iteration from $s=0$ to $S-1$, we have
	\begin{align*}
	u_0\mathbb{E}[{\| {\nabla f({\hat x_k^s })} \|^2}] = \frac{1}{S}\sum\limits_{s = 0}^{S - 1} 
	{\frac{1}{K}\sum\limits_{k = 0}^{K - 1} {u_0\mathbb{E}[{{\| {\nabla f(x_k^s)} \|}^2}]} } +J_0 
	\le 
	\frac{{\mathbb{E}[f({{\tilde x}_0})] - \mathbb{E}[f({{\tilde x}_S})]}}{{KS}}+J_0 \le 
	\frac{{f({x_0}) - f({x^*})}}{{KS}}+J_0,
	\end{align*}
	where $x_k^s$ indicates the $s$-th outer iteration at $k$-th inner iteration, and $\hat 
	x_k^s$ is uniformly and randomly chosen from  $s=\{0,...,S-1\}$ and k=$\{0,..,K-1\}$.
\end{proof}

\textbf{Proof of Corollary \ref{VRNonCS-SCSG:Non:Lemma:paraeters setting}}

\begin{proof}
	In order to have $\mathbb{E}[\| \nabla f(\hat x_k^s ) \|^2]  \le \varepsilon$, that is
	\begin{center}
		$\mathbb{E}[{\| {\nabla f({\hat x_k^s })} \|^2}]   \le \frac{{L_f(f({x_0}) - 
				f({x^*}))}}{{u_0 SK}}+J_0/{u_0} \le \frac{ \varepsilon}{2}+\frac{ \varepsilon}{2} 
		\le  \varepsilon$, 
	\end{center}
	we consider the corresponding parameters' setting:
	\begin{enumerate}
		\item For the first term, consider $c_k$ defined in 
		(\ref{VRNonCS-SCSG:Non:Lemma:NewSequence-c_k})
		define ${c_{k }} = 
		{c_{k+1}}Y + U$,  based on Lemma \ref{VRNonCS-SCSG:Lemma:GeometriProgression}, for 
		$k=K$, 
		we 
		have
		\begin{align*}
		{c_K} = {\left( {\frac{1}{Y}} \right)^K}\left( {{c_0} + \frac{U}{{Y - 1}}} \right) - 
		\frac{U}{{Y - 1}},
		\end{align*}
		where
		\begin{align*}
		Y=& {1 + \left( {\frac{2}{h} + 4hW} \right)\eta  + 10\left( 
			{B_G^4L_F^2 + W} \right){\eta ^2}},\\
		U=& 2W\eta  + 5(L_f^2 + W){L_f}{\eta ^2}>0.
		\end{align*}
		By setting $c_K\to 0$,  we obtain
		\begin{align*}
		{c_0} = \frac{{U{Y^K}}}{{Y - 1}} - \frac{U}{{Y - 1}} = \frac{{U\left( {{Y^K} - 1} 
				\right)}}{{Y 
				- 1}}.
		\end{align*}				
		Then, putting the Y and U  into the above equation. We have
		\begin{align}\label{VRNonCS-SCSG:Non:Lemma:paraeters setting-c0}
		{c_0} = \frac{{2W\eta  + 5(L_f^2 + W){L_f}{\eta ^2}}}{{\left( {\frac{2}{h} + 4hW} 
				\right)\eta  + 
				10\left( {L_f^2 + W} \right){\eta ^2}}}C = \frac{{2W + 5(L_f^2 + W){L_f}{\eta} 
		}}{{\left( 
				{\frac{2}{h} + 4hW} \right) + 10\left( {L_f^2 + W} \right)\eta }}C,			
		\end{align}
		where $C = {Y^K} - 1$. Because 	$c_0$ has the influence on the parameters such as $K$, $C$ 
		and $u_0$, we analyze them separately,	
		\begin{enumerate}
			\item For $K$ and $C$, based on  the  character of function ${\left( {1 + 
					\frac{1}{t_2}} 
				\right)^{t_1}} \to e$,\footnote{Here the 'e' is the Euler number, approximate to 
				2.718. } as $t_1,t_2 \to  + \infty $ and $t_1t_2<1$, and the function 
			is also the increasing	function with 
			an upper bound of $e$, we require		
			\begin{align}\label{VRNonCS-SCSG:Non:Lemma:paraeters setting-K}
			K < 1/\left( {\left( {\frac{2}{h} + 4hW} \right)\eta  + 10\left( {L_f^2 + W} 
				\right){\eta ^2}} \right),
			\end{align}
			thus, we have $C<e-1$.
			\item  For $u_0$ defined in (\ref{VRNonCS-SCSG:Non:Lemma:NewSequence-u_k}), in order to 
			keep $u_k>0$, we need to keep $c_0h<1/4$. If 
			$c_0h<1/4$, there 
			exits a constant $\tilde{u}$ such that $u_0=\tilde{u}\eta$. In order to satisfy 
			$c_0h<1/4$, combine 
			with (\ref{VRNonCS-SCSG:Non:Lemma:paraeters setting-c0}) and $C<e-1$, that is
			\begin{align*}
			{c_0}h \le \frac{{2W + 5(L_f^2 + W){L_f}{\eta} }}{{\left( {\frac{2}{h} + 4hW} 
					\right) + 
					10\left( 
					{L_f^2 + W} \right)\eta }}h\left( {e - 1} \right) \le \frac{1}{4},
			\end{align*}
			
			\begin{enumerate}
				\item 	By setting $h = \frac{1}{{5\sqrt {L_f^3\eta } }}$, there exist 
				$\tilde{w}>0$, based on 
				above inequality, we have			
				\begin{align*}
				W \le \frac{{16L_f^3\eta  + 50L_f^{3.5}\sqrt \eta  \eta }}{{9.6 + 34L_f^3\eta  - 
						50\sqrt {L_f^3\eta } \eta }} < \tilde wL_f^3\eta 
				\end{align*}
				Thus, combine with the definition of W in 
				(\ref{VRNonCS-SCSG:Non:Lemma:NewSequence-W}), we require that
				\begin{align*}
				W =& B_G^4L_F^2\left( {4\frac{\mathbb{I}{\left( {A < n} \right)}}{A} + 
					4\frac{\mathbb{I}{\left( {D < 
								n} \right)}}{D}} \right)\\
				\le& 8B_G^4L_F^2\frac{\mathbb{I}{\left( {A < n} \right)}}{A}
				\le \tilde wL_f^3\eta  = \mathcal{O}\left( L_f^3\eta  \right).
				\end{align*}
				If $A<n$, we  require $A \ge {\cal O}\left( {B_G^4L_F^2/(L_f^3\eta )} \right)$. 
				Thus, we 
				have 
				$A = \min \left\{ {n,\mathcal{O}\left( {1/\eta } \right)} \right\}$.
				\item Based on the setting of $h$ and $W$, combing with 
				(\ref{VRNonCS-SCSG:Non:Lemma:paraeters setting-K}), we have
				\begin{align*}
				K <& \frac{1}{{\left( {10\sqrt {L_f^3\eta }  + \frac{4}{{5\sqrt {L_f^3\eta } 
							}}\tilde wL_f^3\eta } \right)\eta  + 10\left( {L_f^2 + \tilde 
							wL_f^3\eta } 
						\right){\eta ^2}}}\\
				=& \frac{1}{{\left( {10\sqrt {L_f^3\eta }  + \frac{4}{5}\sqrt {L_f^3\eta } } 
						\right)\eta  + 10\left( {L_f^2 + \eta } \right){\eta ^2}}} = {\cal O}\left( 
				{\frac{1}{{{{\left( {{L_f}\eta } \right)}^{3/2}}}}} \right).
				\end{align*}
				
			\end{enumerate}
		\end{enumerate}
		\item For the second term about $J_0$, as $u_0=w_1\eta$, we require
		\begin{align*}
		\frac{{{J_0}}}{{{\tilde{u}}\eta }} = &\frac{1}{{{\tilde{u}}}}\left( {\frac{1}{2} + h{c_0}} 
		\right){W_2} + \left( {{L_f} + 2{c_0}} \right){W_1}\eta \\
		\le& \frac{1}{{{\tilde{u}}}}{W_1}\left( {\frac{3}{5} + {L_f}\eta  + \frac{1}{2}\eta \sqrt 
			\eta  } \right)\\
		\le& \frac{1}{{{\tilde{u}}}}\left( {20B_G^2L_F^2{H_1} + 5{H_2}} \right)\left( {\frac{3}{5} 
			+ 
			{L_f}\eta  + \eta \sqrt \eta  } \right)\frac{\mathbb{I}{(D < n)}}{D}
		\le \frac{1}{2}\varepsilon, 
		\end{align*}
		Then, if $D<n$, we require that 
		\begin{center}
			$D \ge \frac{2}{{\varepsilon {\tilde{u}}}}\left( {20B_G^2L_F^2{H_1} + 5{H_2}} 
			\right)\left( 
			{\frac{3}{5} + \frac{1}{2}{L_f}\eta  + {c_0}\eta \sqrt \eta  } \right) = 
			\mathcal{O}\left( 
			{\frac{1}{\varepsilon }} \right)$.
		\end{center}
		Thus, we set $D = \min \left\{ {n,\mathcal{O}(1/\varepsilon )} \right\}$.
		\item Based on  the first term $\frac{{{L_f}(f({x_0}) - f({x^*}))}}{{\eta SK}} 
		\le \frac{1}{2}\varepsilon $, the total number of iteration is $T=SK=\frac{{2{L_f}(f({x_0}) 
				- f({x^*}))}}{{\eta 
				\varepsilon }}$. 
	\end{enumerate}
	Thus, based on the above parameters' setting, we can ensure that $	\mathbb{E}[\| \nabla f(\hat 
	x_k^s ) \|^2]  \le \varepsilon$.
\end{proof}
\section{Proof for the Mini-batch of the SC-SGSG to the composition problem}

\begin{lemma2}\ref{VRNonCS-SCSG-miniBatch:Lemma:Bound-estimate-G-4}
	Suppose 
	Assumption \ref{VRNonCS-SCSG:Assumption:G}-\ref{VRNonCS-SCSG:Assumption:Middle-Bound} holds, 
	for 
	${{\hat 
			G}_k}$ defined in (\ref{VRNonCS-SCSG:Definition:Estimate-G}) and $\Lambda $ defined in 
	Algorithm \ref{VRNonCS-SCSG:AlgorithmII} with $\mathcal{D} = 
	\left[ 
	{{\mathcal{D}_1},{\mathcal{D}_2}} \right]$ and $D =\left|{\cal 
		D}_1\right|=\left|{\cal 
		D}_2\right|$, we have
	\begin{align*}
	&\mathbb{E}_{i_k,j_k,\cal A, \cal D}{\| {\Lambda  - \nabla f\left( {{x_k}} 
			\right)} \|^2}\\
	\le&5B_G^4L_F^2\left( 
	{\frac{L_f^2}{bB_G^4L_F^2} + 
		4\frac{{\mathbb{I}\left( {A < n} \right)}}{A} + 4\frac{{\mathbb{I}\left( {D < n} 
				\right)}}{D}} \right)\mathbb{E}{\left\| {{x_k} - {{\tilde x}_s}} \right\|^2} + 
	20B_G^2L_F^2\frac{{\mathbb{I}(D < 
			n)}}{D}{H_1} + 5\frac{{\mathbb{I}(D^2 < n^2)}}{{D^2}}{H_2},
	\end{align*}
\end{lemma2}

\begin{proof}
	Through adding and subtracting the term of $\frac{1}{b}\sum\limits_{\left( {i,j} \right) \in 
		{I_b}}^{} {{(\partial {G_j}({x_k}))^\mathsf{T}}\nabla 
		{F_i}(G({x_k}))} 	
	$, $\frac{1}{b}\sum\limits_{\left( {i,j} \right) \in {I_b}}^{} { {(\partial {G_i}({{\tilde 
					x}_s}))^\mathsf{T}}\nabla {F_i}(G({{\tilde 
				x}_s}))}$,  and 
	$(\partial 
	G({{\tilde x}_s}))^\mathsf{T}\nabla F(G({{\tilde x}_s}))$, ${(\partial {G_{{{\cal 
						D}_1}}}({{\tilde 
				x}_s}))^\mathsf{T}}\nabla {F_{{{\cal D}_1}}}(G({{\tilde x}_s}))$, we have
	\begin{align*}
	&\mathbb{E}{\| { \Lambda  - \nabla f\left( {{x_k}} \right)} \|^2}\\
	\le& 5E{\left\| {\frac{1}{b}\sum\limits_{\left( {i,j} \right) \in {I_b}}^{} {{{(\partial 
						{G_j}({x_k}))}^T}\nabla {F_i}(G({x_k})) - {{(\partial {G_j}({{\tilde 
						x}_s}))}^T}\nabla 
				{F_i}(G({{\tilde x}_s}))}  - \left( {\nabla f({x_k}) - {{(\partial G({{\tilde 
				x}_s}))}^T}\nabla 
				F(G({{\tilde x}_s}))} \right)} \right\|^2}\\
	&+ 5E{\left\| {\frac{1}{b}\sum\limits_{\left( {i,j} \right) \in {I_b}}^{} {{{(\partial 
						{G_j}({x_k}))}^T}\nabla {F_i}({{\hat G}_k}) - {{(\partial 
						{G_j}({x_k}))}^T}\nabla {F_i}(G({x_k}))} 
		} \right\|^2}\\
	&+ 5E{\left\| {\frac{1}{b}\sum\limits_{\left( {i,j} \right) \in {I_b}}^{} {{{(\partial 
						{G_j}({{\tilde x}_s}))}^T}\nabla {F_i}(G({{\tilde x}_s})) - {{(\partial 
						{G_j}({{\tilde 
								x}_s}))}^T}\nabla {F_i}({G_{{D_1}}}({{\tilde x}_s}))} } 
								\right\|^2}\\
	&+ 5E{\left\| {\nabla {{\hat f}_{\cal D}}({{\tilde x}_s}) - {{(\partial {G_{{D_1}}}({{\tilde 
							x}_s}))}^T}\nabla {F_{{D_2}}}(G({{\tilde x}_s}))} \right\|^2}\\
	&+ 5E{\left\| {{{(\partial {G_{{\mathcal{D}_1}}}({{\tilde x}_s}))}^T}\nabla 
	{F_{{\mathcal{D}_2}}}(G({{\tilde 
	x}_s})) - 
			{{(\partial G({{\tilde x}_s}))}^T}\nabla F(G({{\tilde x}_s}))} \right\|^2}\\
	\mathop  \le \limits^{\scriptsize \textcircled{\tiny{2}}}& {\frac{5}{b}}L_f^2\mathbb{E}{\left\| 
		{{x_k} - 
			{{\tilde x}_s}} \right\|^2} + 5B_G^2L_F^2\mathbb{E}{\left\| {{{\hat 
					G}_k} - G({x_k})} \right\|^2} + 5B_G^2L_F^2\mathbb{E}{\left\| {G({{\tilde 
					x}_s}) - {G_{{{\cal 
							D}_1}}}({{\tilde x}_s})} \right\|^2}\\& + 5B_G^2L_F^2\mathbb{E}{\left\| 
		{G({{\tilde x}_s}) - 
			{G_{{{\cal 
							D}_1}}}({{\tilde x}_s})} \right\|^2} + 5\frac{{\mathbb{I}(D^2 < 
			n^2)}}{{D^2}}{H_2}\\
	\mathop  \le \limits^{\scriptsize \textcircled{\tiny{3}}}& 5B_G^4L_F^2\left( 
	{\frac{L_f^2}{B_G^4L_F^2} + 
		4\frac{{\mathbb{I}\left( {A < n} \right)}}{A} + 4\frac{{\mathbb{I}\left( {D < n} 
				\right)}}{D}} \right)\mathbb{E}{\left\| {{x_k} - {{\tilde x}_s}} \right\|^2} + 
	20B_G^2L_F^2\frac{{\mathbb{I}(D < 
			n)}}{D}{H_1} + 5\frac{{\mathbb{I}(D^2 < n^2)}}{{D^2}}{H_2},
	\end{align*}	
	where ${\small\textcircled{\scriptsize{1}}}$ follows from $||a_1+a_2+a_3+a_4+a_5||^2\le 
	5a_1^2+5a_2^2+5a_3^2+5a_4^2+5a_5^2$, and Lemma 
	\ref{VRNonCS-SCSG:Appendix:Lemma:Inequation-indepent},
	${\small\textcircled{\scriptsize{2}}}$ is based on $
	\mathbb{E}[ \| X - \mathbb{E}[ X ] \|^2 ]  =  \mathbb{E}[ X^2 - \| \mathbb{E}[ 
	X ] \|^2 ] \le \mathbb{E}[ X^2 ]$, the smoothness of $F_i$ in Assumption 
	\ref{VRNonCS-SCSG:Assumption:GF}, the bounded Jacobian of $G(x)$ and the 
	smoothness of $F$ in Assumption \ref{VRNonCS-SCSG:Assumption:G} and 
	\ref{VRNonCS-SCSG:Assumption:F}, and the upper 
	bound of variance  in Assumption 
	\ref{VRNonCS-SCSG:Assumption:Middle-Bound} and Lemma 
	\ref{VRNonCS-SCSG:Appendix:Lemma:Inequation-indepent-double}.
	${\small\textcircled{\scriptsize{3}}}$ is based on Lemma 
	\ref{VRNonCS-SCSG:Lemma:Bound-estimate-G-2} and Assumption 
	\ref{VRNonCS-SCSG:Assumption:Middle-Bound}.
\end{proof}

\begin{corollary2}\ref{VRNonCS-SCSG-miniBatch-non:query complexity}
	Suppose Assumption \ref{VRNonCS-SCSG:Assumption:G}-
	\ref{VRNonCS-SCSG:Assumption:Middle-Bound} holds, in Algorithm \ref{VRNonCS-SCSG:AlgorithmII}, 
	Let $h =\sqrt {b/\eta}$,  	the step size is  $\eta  = {b^{3/5}}\min \{ 1/n^{2/5},\varepsilon 
	^{2/5} 
	\}$, the set-size 
	of $\cal A$ is $A = \min \left\{ {n,\mathcal{O}\left( {b/\eta } \right)} \right\}$, the 
	set-size of the subset 
	$\mathcal{D}_1$ and 
	$\mathcal{D}_2$ are $ D = \min \left\{ {n,\mathcal{O}(1/\varepsilon )} \right\}$,  the number 
	of 
	inner iteration is $K \le 
	\mathcal{O}\left( {b^{1/2}/({\eta ^{3/2}})} \right)$, the total number of iteration is $T = 
	\mathcal{O}\left( {1/\left( 
		{\varepsilon \eta } \right)} \right)$, in order to obtain 
	obtain	$\mathbb{E}[\| \nabla f(\hat x_k^s ) \|^2]  \le 
	\varepsilon$.The query complexity is 
	\begin{align*}
	 \frac{1}{{{b^{1/5}}}}\mathcal{O}\left( {\min \left\{ {\frac{1}{{{\varepsilon 
						^{9/5}}}},\frac{{{n^{4/5}}}}{\varepsilon }} \right\}} \right)
	\end{align*}
\end{corollary2}
\begin{proof}
	Based on the parameters' setting, that is $ D = \min \left\{ {n,\mathcal{O}(1/\varepsilon )} 
	\right\}$, $A = \min \left\{ {n,\mathcal{O}\left( {b/\eta } \right)} \right\}$, $K \le 
	\mathcal{O}\left( {b^{1/2}/{{\eta ^{3/2}}}} \right)$, and  $T = 
	\mathcal{O}\left( {1/\left( 
		{\varepsilon \eta } \right)} \right)$, we have,
	\begin{align*}
	{\cal O}\left( {\frac{T}{K}\left( {D + KA} \right)} \right) =& 
	 O\left( {\frac{{{\eta ^{3/2}}}}{{\varepsilon {b^{1/2}}\eta }}\left( {\min \left\{ 
	 {n,\frac{1}{\varepsilon }} \right\} + \frac{{{b^{1/2}}b}}{{{\eta ^{3/2}}\eta }}} \right)} 
	 \right) = O\left( {\frac{{{\eta ^{1/2}}}}{{\varepsilon {b^{1/2}}}}\left( {\min \left\{ 
	 {n,\frac{1}{\varepsilon }} \right\} + \frac{{{b^{3/2}}}}{{{\eta ^{5/2}}}}} \right)} \right)\\
	=& \frac{1}{{\varepsilon {b^{1/2}}}}O\left( {\min \left\{ {n,\frac{1}{\varepsilon }} 
	\right\}{\eta ^{1/2}} + \frac{{{b^{3/2}}}}{{{\eta ^2}}}} \right)\\
	\ge& \frac{1}{{{b^{1/5}}}}O\left( {\min \left\{ {\frac{{{n^{4/5}}}}{\varepsilon 
	},\frac{1}{{{\varepsilon ^{9/5}}}}} \right\}} \right)
	\end{align*}
	where the optimal $\eta  = {b^{3/5}}\min \left\{ {1/{n^{2/5}},{\varepsilon ^{2/5}}} \right\}$.
\end{proof}

\end{document}